\documentclass[a4paper,12pt]{article}
\usepackage{colortbl}
\usepackage{amssymb,latexsym,amsmath,enumerate,verbatim,amsfonts,amsthm}
\usepackage[active]{srcltx}

\newtheorem{lemma}{Lemma}[section]
\newtheorem{definition}{Definition}[section]

\newtheorem{theorem}{Theorem}[section]
\newtheorem{corollary}{Corollary}[section]
\newtheorem{proposition}{Proposition}[section]
\newtheorem{remark}{Remark}[section]
\newtheorem{example}{Example}[section]

\begin{document}
\title{\Large{\sf{\sf{ Trust--Region Problems with Linear Inequality Constraints: Exact SDP Relaxation, Global Optimality  and Robust Optimization\thanks{The authors are grateful to the referees for their valuable suggestions and helpful comments which have contributed to the final preparation of the paper. Research was partially supported by a grant from the Australian Research Council.}}}}}

\date{\today}

\author{
 \ \textsc{V. Jeyakumar}\thanks{Department of Applied Mathematics, University of New South Wales,
Sydney 2052, Australia. E-mail: v.jeyakumar@unsw.edu.au}\ \ and \  \textsc{G. Y.
Li}\thanks{Department of Applied Mathematics, University of New South
Wales, Sydney 2052, Australia. E-mail: g.li@unsw.edu.au} }
\date{Revised Version: \today}
\maketitle

%opening

\maketitle

\begin{abstract}
The trust-region problem, which minimizes a nonconvex quadratic function over a ball, is a key subproblem in trust-region methods for solving nonlinear optimization problems. It enjoys many attractive properties such as an exact semi-definite linear programming relaxation (SDP-relaxation) and strong duality. Unfortunately, such properties do not, in general,
 hold for an extended trust-region problem having extra linear constraints. This paper shows that two useful and powerful features of the classical
 trust-region problem continue to hold for an extended trust-region problem with linear inequality constraints under a new dimension condition. First, we establish
  that the class of extended trust-region problems has an exact SDP-relaxation, which holds without the Slater constraint qualification. This is achieved
 by proving that a system of quadratic and affine functions involved in the  model satisfies a range-convexity whenever the dimension condition is fulfilled.
 Second, we show that the dimension condition together with the Slater condition ensures that a set of combined first and second-order Lagrange
 multiplier conditions is necessary and sufficient for global optimality of the extended trust-region problem and consequently for strong duality. Through simple examples we also provide an insightful account of our development from SDP-relaxation to strong duality. Finally, we show that the dimension condition
is easily satisfied for the extended trust-region model that arises from the reformulation of a robust least squares problem (LSP) as well as a robust second order cone programming model problem (SOCP) as an equivalent semi-definite linear programming problem. This leads us to conclude that, under mild assumptions, solving a robust (LSP) or (SOCP) under matrix-norm uncertainty or polyhedral uncertainty is equivalent to solving a semi-definite linear programming problem and so, their solutions can be validated in polynomial time.
\medskip

%\noindent\textbf{Key words.} Extended trust-regions problems, exact semi-definite programming relaxations, necessary and sufficient global optimality, strong duality, second-order cone programming problems.
%\medskip

%\noindent
%\textbf{AMS subject classification}. 90C20,90C30,90C26,90C46

\end{abstract}
\newpage
\section{Introduction}
Consider the extended trust-region model problem with linear inequality constraints
\begin{eqnarray*}
(P) & \min_{x \in \mathbb{R}^n}  &  x^TAx+a^Tx \\
& \mbox{ s.t. } & \|x-x_0\|^2 \le \alpha, \\
& & b_i^Tx \le \beta_i, \, i=1,\ldots,m,
\end{eqnarray*}
where $A$ is a symmetric $(n \times n)$ matrix, $a,b_i,x_0 \in \mathbb{R}^n$ and $\alpha,\beta_i \in \mathbb{R}, \alpha > 0$, $i=1,\ldots,m$. Model problems of this form arise from the application of the
trust region method for solving constrained optimization problems \cite{trust0}, such as nonlinear programming problems with linear inequality constraints, nonlinear optimization problems with discrete variables \cite{kurt,pardalos} (see Section 2) and robust optimization problems \cite{Bert-survey,robust_book} under matrix norm \cite{bert} or polyhedral uncertainty \cite{robust_book,jl_siam,jl_ORL2} (see Section 5). The model (P) with a single linear inequality constraint, where $m=1$ and $x_0=0$, has recently been examined in the literature (see \cite{kurt,beck-eldar} and other references therein).
\medskip

In the special case of (P) where $(b_i, \beta_i)=(0, 0)$, it is the well-known trust-region model, and it has been extensively studied from both theoretical and algorithmic points of view \cite{trust-1,Henry,powel,Yuan_MP}. The classical trust-region problem enjoys exact semi-definite programming relaxation (SDP-relaxation) and admits strong duality. Moreover, its solution can be found by solving a dual Lagrangian system. Unfortunately, these results are, in general, no longer true for our extended trust-region model (P). Indeed, even in the simplest case of (P) with a single linear inequality constraint, it has been shown that the SDP-relaxation is not exact (see \cite{beck-eldar,Shuzhong} and other references therein). However, in the case of single inequality constraint, exact SDP-relaxation and strong duality hold under a dimension condition (see \cite{beck-eldar} and Corollary 4.2 in Section 4).

In this paper, we make the following contributions which extend the attractive features of the classical trust-region model to our extended trust-region model (P) under a new dimension condition:
\begin{itemize}

\item[{\bf (i)}] Exploiting a hidden convexity property of the extended trust-region system of (P), we establish that the SDP-relaxation of our extended trust-region problems (P)
 is exact whenever a dimension condition is fulfilled.

The dimension condition requires that the number of inequalities must be strictly less than the multiplicity of the minimum eigenvalue of the matrix $A$. It guarantees a joint-range-convexity for the extended trust-region system of (P).

The exact SDP-relaxation is derived without the standard Slater condition. For related  exact relaxation result for problems involving uniform quadratic systems (see \cite{Beck0} and other reference therein).

\item[{\bf (ii)}] We present a necessary and sufficient condition for global optimality for our model problem (P). Consequently, we derive strong duality between (P) and its Lagrangian dual problem under the Slater condition. Also, we obtain two forms of S-lemma for extended trust-region systems. In the case of (P) with two linear (bound) constraints our result provides a more general dimension condition than the corresponding condition, given recently in \cite{beck-eldar}.

\item[{\bf (iii)}] Under suitable, but commonly used, uncertainty sets of robust optimization, we show that the dimension condition
is easily satisfied for our extended trust-region model that arises from the reformulation of a robust least squares model problem (LSP) as well as a second order cone programming model problem (SOCP) as a semi-definite linear programming problem. As a result, we establish a complete characterization of the solution of a robust (LSP) and a robust (SOCP) in terms of the solution of a semi-definite linear programming problem.

\end{itemize}

The significance of our contributions is that:
\begin{itemize}
\item[{\bf (i)}] Our dimension condition, expressed in terms of original data, not only reveals a hidden convexity of the extended trust-region problems but also allows direct applications to solving robust optimization problems such as the robust (LSP) and (SOCP) models. These models are increasingly becoming the models of choice for efficiently solving many classes of hard problems by relaxation or reformulation techniques \cite{Second_order,bn,robust_book}.

\item[{\bf (ii)}] Our results show that a worst-case solution of a least-squares problems or a second-order cone programming problem in the face of data uncertainty, especially in the case of a matrix-norm uncertainty or a polyhedral uncertainty, can be found by solving a semi-definite linear programming problem.

\item[{\bf (iii)}] Our approach suggests further extensions of global optimality, strong duality and exact SDP-relaxation results to broad classes of extended trust-region models with (uniformly) convex quadratic constraints by way of examining joint-range convexity properties of the corresponding systems (see Section 6).

\end{itemize}

The outline of the paper is as follows. In Section 2, we introduce the dimension condition and establish a joint-range convexity property. In Section 3 we derive exact relaxation results for (P) and  illustrate the results with numerical examples. In Section 4, we show that the dimension condition together with the Slater condition ensures that a combined first and second-order Lagrange multiplier condition is necessary and sufficient for global optimality of (P) and guarantees strong duality between (P) and its Lagrangian dual. In Section 5, we  present an application of strong duality to S-lemma and consequently to robust optimization problems \cite{robust_book}. In Section 6, we show how our dimension condition can be extended to obtain corresponding exact relaxation and strong duality results for  trust-regions problems with certain convex quadratic inequalities. Finally, in Appendix, for the sake of self-containment, we describe some useful technical results  that are related to non-convex quadratic 
systems and robust optimization.

%{\bf To avoid triviality, we always assume that  $\{x:\|x-x_0\|^2 \le \alpha, b_i^Tx \le \beta_i, \, i=1,\ldots,m\} \neq \emptyset.$}
\section{Hidden Convexity of Extended Trust Regions}
In this section, we derive an important hidden convexity property of extended trust-region quadratic systems which will play a key role
in our study of exact relaxation and strong duality later on.

We begin by fixing the notation and definitions that will be used later in the paper. The
real line is denoted by $\mathbb{R}$ and the $n$-dimensional
real Euclidean space is denoted by $\mathbb{R}^n$. The set of all
non-negative vectors of $\mathbb{R}^n$ is denoted by
$\mathbb{R}^n_{+}.$ The space of all $(n \times
n)$ symmetric real matrices  is denoted by $S^{n\times n}$. The $(n \times n)$
identity matrix is denoted by $I_n$. The notation $A \succeq B$
means that the matrix $A-B$ is positive semi-definite. Moreover, the
notation $A \succ B$ means the matrix $A-B$ is positive definite. The set consists of all $n\times n$ positive semidefinite matrices is denoted by $S^n_+$. Let $A,B \in S^{n\times n}$. The (trace) inner product of $A$
and $B$ is defined by $A \cdot
B=\sum_{i=1}^{n}\sum_{j=1}^{n}a_{ij}b_{ji}$, where $a_{ij}$ is
the $(i,j)$ element of $A$ and $b_{ji}$ is the $(j,i)$ element of
$B$. A useful fact about the trace inner product is that $A \cdot
(xx^T)=x^TAx$ for all $x \in \mathbb{R}^n$ and $A \in S^{n\times n}$.
For a matrix $A \in S^{n \times n}$, ${\rm Ker}(A):=\{d \in \mathbb{R}^n: Ad=0\}$. For a subspace $L$, we use ${\rm dim}\, L$ to denote the dimension of $L$.

As in \cite{Beck0,beck-eldar}, the study of exact relaxation and strong duality requires the examination of topological and geometrical properties of the set
$$U(f,g_0,g_1,\ldots,g_m):=\{(f(x),g_0(x),g_1(x),\ldots,g_m(x)):x \in \mathbb{R}^n\}+\mathbb{R}_+^{m+2},$$ where
$f(x)=x^TAx+a^Tx+\gamma$, $g_0(x)=\|x-x_0\|^2-\alpha$ and $g_i(x)=b_i^Tx-\beta_i$, $i=1,\ldots,m$, $A \in S^{n \times n}$, $a,x_0,b_i \in \mathbb{R}^n$ and $\gamma,\alpha,\beta_i \in \mathbb{R}$, $i=1,\ldots,m$.

We note that the range set $U(f,g_0,g_1,\ldots,g_m)$ is the sum of the nonnegative orthant and the image of the quadratic mapping $\{(f(x),g_0(x),g_1(x),\ldots,g_m(x)):x \in \mathbb{R}^n\}$. Hence,
the range set $U(f,g_0,g_1,\ldots,g_m)$ is convex whenever the image of the quadratic mapping  is convex.
It is known that the joint-range convexity of quadratic mappings has a close relationship with strong duality of an associated optimization problem. For example,
Fradkov and Yakubovich \cite{FY,Yakubovich} used convexity of the joint-range $\{(f(x),g_0(x)):x \in \mathbb{R}^n\}$ in the case of homogeneous (not necessarily convex) quadratic functions $f,g_0$ (cf. \cite{Dine}) to show that strong duality holds for quadratic optimization problem with single quadratic constraint, under the Slater condition.

Recently, Polyak \cite{Polyak} established a strong duality result for homogenous nonconvex quadratic problems involving two quadratic
constraints by showing that the joint-range of three homogenous quadratic functions is convex under a positive definiteness condition. On the other hand, the image of  three nonhomogeneous quadratic function is, in general, not convex.   See \cite{Beck0,S_lemma,Polyak} for more detailed discussion for joint-range convexity of quadratic functions.

We begin by showing that the set $U(f,g_0,g_1,\ldots,g_m)$ is always closed.

\begin{proposition}\label{prop:2.1}
Let $f(x)=x^TAx+a^Tx+\gamma$, $g_0(x)=\|x-x_0\|^2-\alpha$ and $g_i(x)=b_i^Tx-\beta_i$, $i=1,\ldots,m$, $A \in S^{n \times n}$, $a,x_0,b_i \in \mathbb{R}^n$ and $\gamma,\alpha,\beta_i \in \mathbb{R}$, $i=1,\ldots,m$. Then
$U(f,g_0,g_1,\ldots,g_m)$ is  closed.
\end{proposition}
\begin{proof}
Let $(r^k,s_0^k,s_1^k,\ldots,s_m^k) \in U(f,g_0,g_1,\ldots,g_m)$ with $$(r^k,s_0^k,s_1^k,\ldots,s_m^k) \rightarrow (r,s_0,s_1,\ldots,s_m).$$
By the definition, for each $k$, there exists $x^k \in \mathbb{R}^n$
\begin{equation}\label{eq:00}
f(x^k) \le r^k, \ \|x^k-x_0\|^2 \le \alpha+s_0^{k}, \ b_1^Tx^k \le \beta_1+s_1^k \,,  \ldots, \, b_m^T x^k \le \beta_m+s_m^k.
\end{equation}
This implies that $x^k$ is bounded, and so, by passing to subsequences, we may assume that $x_k \rightarrow x$. Then, passing limits in (\ref{eq:00}), we have
\[
f(x) \le r, \|x-x_0\|^2 \le \alpha+s_0, b_1^Tx \le \beta_1+s_1,\ldots,b_m^Tx \le \beta_m+s_m.
\]
That is to say, $(r,s_0,s_1,\ldots,s_m) \in U(f,g_0,g_1,\ldots,g_m)$. So, $U(f,g_0,g_1,\ldots,g_m)$ is closed.
\end{proof}
The following simple one-dimensional example shows that the set $U(f,g_0,g_1,\ldots,g_m)$ is, in general, not a convex set.
\begin{example} \label{ex:0} {\bf (Nonconvexity of $U(f,g_0,g_1,\ldots,g_m)$)}
For (P), let $n=1$, $m=1$, $f(x)=x-x^2$, $g_0(x)=x^2-1$ and $g_1(x)=-x$. Then, $f(x)=x^TAx+a^Tx+r$ with $A=-1$, $a=1$ and $r=0$, $g_0(x)=\|x-x_0\|^2-\alpha$ with $x_0=0$, $\alpha=1$ and $g_1(x)=b_1^Tx-\beta_1$  with $b_1=1$ and  $\beta_1=0$.

Then, the set $U(f,g_0,g_1)$ is not a convex set. To see this, note that $f(0)=0$, $g_0(0)=-1$ and $g_1(0)=0$, and $f(1)=0$, $g_0(1)=0$ and $g_1(1)=-1$. So, $(0,-1,0) \in U(f,g_0,g_1)$ and $(0,0,-1) \in U(f,g_0,g_1)$. However, the mid point $(0,-\frac{1}{2},-\frac{1}{2}) \notin U(f,g_0,g_1)$.
Otherwise, there exists $x \in \mathbb{R}$ such that
\[
x-x^2 \le 0,  x^2-1 \le -\frac{1}{2} \mbox{ and } -x \le -\frac{1}{2}.
\]
It is easy to check that the above inequality system has no solution. This is a contradiction, and hence $(0,-\frac{1}{2},-\frac{1}{2}) \notin U(f,g_0,g_1)$. Thus,  $U(f,g_0,g_1)$ is not convex.

%Finally, it is worth noting that this inequality system does not satisfy the following dimension condition
%\[
%{\rm dim}{\rm Ker}(A -\lambda_{\rm min}(A)I_n)+m \chi(\lambda_{\rm min}(A))+{\rm dim} b_1^{\bot}\ge 2=n+1,
%\]
%where $b_1^{\bot}=\{x: b_1^Tx=0\}$.
%
%   $g_1(x) \le 0$ and $g_2(x) \le 0$ if and only if $0 \le x \le 1$, and $\min_{x \in [0,1]}\{x-x^2 \}= 0.$ So, $g_1(x) \le 0, g_2(x) \le 0 \Rightarrow f(x) \ge 0$. However, there is no $\lambda_1,\lambda_2 \ge 0$ such that $f(x)+\lambda_1g_1(x)+\lambda_2 g_2(x) \ge 0$ for all $x \in \mathbb{R}$. To see this, we proceed by the method of contradiction and let $\lambda_1,\lambda_2 \ge 0$ such that $f(x)+\lambda_1g_1(x)+\lambda_2 g_2(x)=(-1+\lambda_1+\lambda_2)x^2+(1-2\lambda_2)x-\lambda_1 \ge 0$
% for all $x \in \mathbb{R}$. Letting $x=0$, we have $\lambda_1 \le 0$, and so, $\lambda_1=0$. So, $(-1+\lambda_2)x^2+(1-2\lambda_2)x \ge 0$ for all $x \in \mathbb{R}$. This implies that $-1+\lambda_2 \ge 0$ (by letting $x \rightarrow \infty$) and $1-2\lambda_2\ge0$ (by letting $x \rightarrow 0$) which is impossible. Thus, the conclusion of Theorem \ref{prop:0} fails.
\end{example}

The following dimension condition plays a key role in the rest of the paper.
Recall that, for a matrix $A\in S^n$, $\lambda_{\rm min}(A)$ denotes the smallest eigenvalue of $A$.
\begin{definition}{\bf (Dimension condition)} Consider the system of functions
$f(x)=x^TAx+a^Tx+\gamma$, $g_0(x)=\|x-x_0\|^2-\alpha$ and $g_i(x)=b_i^Tx-\beta_i$, $i=1,\ldots,m$, where $A \in S^{n \times n}$, $a,x_0,b_i \in \mathbb{R}^n$ and $\gamma,\alpha,\beta_i \in \mathbb{R}$. Let
${\rm dim}~{\rm span}\{b_1,\ldots,b_m\}=s$, $s \le n$. Then, we say that the dimension condition holds for the system whenever
\begin{equation}\label{eq:DC}
 {\rm dim}~{\rm Ker}(A -\lambda_{\rm min}(A)I_n) \ge s +1.
\end{equation}
\end{definition}
In other words, the dimension condition states that the multiplicity of the minimum
eigenvalue of $A$ is at least $s+1$.

 Recall that the optimal value function $h:\mathbb{R}^{m+1}\rightarrow \mathbb{R} \cup \{+ \infty\}$ of (P) is given by
\begin{eqnarray*}
 & & h(r,s_1,\ldots,s_m) \\
&= & \left\{ \begin{array}{ccc}
\displaystyle \min_{x \in \mathbb{R}^n}\{f(x): \|x-x_0\|^2 \le \alpha+r, \, b_i^Tx \le \beta+s_i,\, i=1,\ldots,m\}, & (r,s_1,\ldots,s_m) \in D, & \\
+\infty, & \mbox{otherwise,} &
            \end{array}
\right.
\end{eqnarray*}
where $D=\{(r,s_1,\ldots,s_m): \|x-x_0\|^2 \le \alpha+r, b_i^Tx \le \beta_i+s_i \mbox{ for some } x \in \mathbb{R}^n\}$.

\begin{theorem}{\bf (Dimension condition revealing hidden convexity)}\label{prop:2.2}
Let $f(x)=x^TAx+a^Tx+\gamma$, $g_0(x)=\|x-x_0\|^2-\alpha$ and $g_i(x)=b_i^Tx-\beta_i$, $i=1,\ldots,m$, $A \in S^{n \times n}$, $a,x_0,b_i \in \mathbb{R}^n$ and $\gamma,\alpha,\beta_i \in \mathbb{R}$.
Suppose that the dimension condition (\ref{eq:DC}) is satisfied.
%\[
%{\rm dim}{\rm Ker}(A -\lambda_{\rm min}(A)I_n)+{\rm dim}(\bigcap_{i=1}^mb_i^{\bot}) \ge n +1 .
%\]
Then, $$U(f,g_0,g_1,\ldots,g_m):=\{(f(x),g_0(x),g_1(x),\ldots,g_m(x)):x \in \mathbb{R}^n\}+\mathbb{R}_+^{m+2}$$ is a convex set.
\end{theorem}
\begin{proof}

%[Closedness] Let $(r^k,s_0^k,\ldots,s_m^k) \in U(f,g_0,g_1,\ldots,g_m)$ with $$(r^k,s_0^k,\ldots,s_m^k) \rightarrow (r,s_0,\ldots,s_m).$$
%By the definition, for each $k$, there exist $x^k \in \mathbb{R}^n$
%\begin{equation}\label{eq:00}
%f(x^k) \le r^k, \ \|x^k\|^2 \le s_0^{k}, \ b_1^T(x^k) \le s_1^k \,,  \ldots, \, b_m^T x^k \le s_m^k.
%\end{equation}
%This implies that $x^k$ is bounded, and so, by passing to subsequences, we may assume that $x_k \rightarrow x$. Then, passing limits in (\ref{eq:00}), we have
%\[
%f(x) \le r, \|x\|^2 \le s_0, b_1^Tx \le s_1,\ldots,b_m^Tx \le s_m.
%\]
%That is to say, $(r,s_0,s_1,\ldots,s_m) \in U(f,g_0,g_1,\ldots,g_m)$. So, $U(f,g_0,g_1,\ldots,g_m)$ is closed.
%
%[Convexity]
%Note that,
%if $\lambda_{\rm min}(A) \ge 0$, then $U(f,g_0,g_1,\ldots,g_m)$ is convex  as $f,g_i$ are all convex.
%So, we may assume that $\lambda_{\rm min}(A)<0$. Then, our assumption collapses to
%\[
%{\rm dim}{\rm Ker}(A -\lambda_{\rm min}(A)I_n)+{\rm dim}(\bigcap_{i=1}^mb_i^{\bot}) \ge n +1 .
%\]

We first note that, if $A$ is positive semidefinite, then $f,g_i$, $i=0,1,\ldots,m$, are all convex functions. So, $U(f,g_0,g_1,\ldots,g_m)$ is always convex in this case. Therefore,
we may assume that $A$ is not positive semidefinite and hence $\lambda_{\rm min}(A)<0$.

{\bf [${\bf U(f,g_0,g_1,\ldots,g_m)=\mbox{epi}h}$]}.
Let $D=\{(r,s_1,\ldots,s_m): \|x-x_0\|^2 \le \alpha+r, b_i^Tx \le \beta_i+s_i \mbox{ for some } x \in \mathbb{R}^n\}$. Clearly, $D$ is a convex set.
Then, by the definition, we have
$U(f,g_0,g_1,\ldots,g_m) %& = & \{(t,r,s_1,\ldots,s_m): f(x) \le t, \|x\|^2 \le r, b_i^Tx \le s_i \mbox{ for some } x \in \mathbb{R}^n\} \\
={\rm epi}h.$
\medskip

\noindent
{\bf [Convexity of the value function $h$]}. To see this,  we claim that, for each $(r,s_1,\ldots,s_m) \in D$, the minimization problem  $$\min_{x \in \mathbb{R}^n}\{f(x)-\lambda_{\min}(A)\|x-x_0\|^2: \|x-x_0\|^2 \le \alpha+r, b_i^Tx \le \beta_i+s_i\}$$ attains its minimum at some $\overline{x} \in \mathbb{R}^n$ with $\|\overline{x}-x_0\|^2=\alpha+r$ and $b_i^T\overline{x} \le \beta_i+s_i$. Granting this, we have
\begin{eqnarray*}
 & & \min_{x \in \mathbb{R}^n}\{f(x)-\lambda_{\min}(A)\|x-x_0\|^2: \|x-x_0\|^2 \le \alpha+r, \, b_i^Tx \le \beta_i+s_i\}\\
& = & f(\overline{x})- \lambda_{\min}(A) (\alpha+r)  \\
& \ge & \min_{x \in \mathbb{R}^n}\{f(x): \|x-x_0\|^2 \le \alpha+r, \, b_i^Tx \le \beta+s_i\}- \lambda_{\min}(A) (\alpha+r) \\
& = & \min_{x \in \mathbb{R}^n}\{f(x)- \lambda_{\min}(A) (\alpha+r) : \|x-x_0\|^2 \le \alpha+r, \, b_i^Tx \le \beta+s_i\} \\
& \ge &  \min_{x \in \mathbb{R}^n}\{f(x)-\lambda_{\min}(A)\|x-x_0\|^2: \|x-x_0\|^2 \le \alpha+r, \, b_i^Tx \le \beta_i+s_i\},
\end{eqnarray*}
where the last inequality follows by $\lambda_{\min}(A)<0$. This yields that
\begin{eqnarray*}
& & \min_{x \in \mathbb{R}^n}\{f(x): \|x-x_0\|^2 \le \alpha+r, \, b_i^Tx \le \beta+s_i,\, i=1,\ldots,m\} \\
& = & \min_{x \in \mathbb{R}^n}\{f(x)-\lambda_{\min}(A)\|x-x_0\|^2: \|x-x_0\|^2 \le \alpha+r, \, b_i^Tx \le \beta_i+s_i\}+\lambda_{\min}(A) (\alpha+r).
\end{eqnarray*}
Note that
\[
F(x):=f(x)-\lambda_{\min}(A)\|x-x_0\|^2= x^T(A-\lambda_{\rm min}(A)I_n)x+(a+2\lambda_{\min}(A)x_0)^Tx+(\gamma-\lambda_{\min}(A)\|x_0\|^2)
\]
is a convex function, and so, $(r,s_1,\ldots,s_m) \mapsto \min_{x \in \mathbb{R}^n}\{F(x): \|x-x_0\|^2 \le \alpha+r, b_i^Tx \le \beta_i+s_i\}$ is also convex. It follows that
\[
(r,s_1,\ldots,s_m) \mapsto   \min_{x \in \mathbb{R}^n}\{f(x): \|x-x_0\|^2 \le \alpha+r, \, b_i^Tx \le \beta+s_i,\, i=1,\ldots,m\}
\]
is convex. Therefore, $h$ is convex, and so, $U(f,g_0,g_1,\ldots,g_m)={\rm epi}h$ is a convex set.
\medskip

\noindent
{\bf [Attainment of minimizer on the sphere]} To see the claim, we proceed by the method of contradiction and suppose that any minimizer $x^*$ of
$$\min_{x \in \mathbb{R}^n}\{F(x): \|x-x_0\|^2 \le \alpha+r, b_i^Tx \le \beta_i+s_i\}$$
satisfy $\|x^*-x_0\|^2<\alpha+r$ and $b_i^Tx^* \le \beta_i+s_i$. We note that there exists $v \in \mathbb{R}^n \backslash\{0\}$ such that
\begin{equation}\label{eq:0}
v \in \big( \bigcap_{i=1}^mb_i^{\bot} \big) \cap {\rm Ker}(A -\lambda_{\rm min}(A)I_n).
\end{equation}
[Otherwise, $\big( \bigcap_{i=1}^mb_i^{\bot} \big) \cap {\rm Ker}(A -\lambda_{\rm min}(A)I_n)=\{0\}$. Recall from our dimension condition that ${\rm dim}{\rm Ker}(A -\lambda_{\rm min}(A)I_n) \ge s+1$ where $s$ is the dimension of ${\rm span}\{b_1,\ldots,b_m\}$.  Then, it follows from the dimension theorem that
\begin{eqnarray*}
n+1 & =&  (s+1) +(n-s) \\
& \le & {\rm dim}{\rm Ker}(A -\lambda_{\rm min}(A)I_n)+{\rm dim}(\bigcap_{i=1}^mb_i^{\bot}) \\
& = & {\rm dim}\big({\rm Ker}(A -\lambda_{\rm min}(A)I_n)+\bigcap_{i=1}^mb_i^{\bot} \big)+{\rm dim} \big( \bigcap_{i=1}^mb_i^{\bot}  \cap {\rm Ker}(A -\lambda_{\rm min}(A)I_n) \big)\\
& \le &  n,
\end{eqnarray*}
which is impossible, and hence (\ref{eq:0}) holds.]
Fix an arbitrary minimizer $x^*$ of
$\min_{x \in \mathbb{R}^n}\{F(x): \|x-x_0\|^2 \le \alpha+r, b_i^Tx \le \beta_i+s_i\}$. We now split the discussion into two cases:
Case 1, $(a+2\lambda_{\min}(A)x_0)^Tv=0$; \ Case 2, $(a+2\lambda_{\min}(A)x_0)^Tv \neq 0$.

Suppose that case 1 holds, i.e., $(a+2\lambda_{\min}(A)x_0)^Tv=0$. Consider $x(t)=x^*+tv$. As $\|x^*-x_0\|^2 < \alpha+r$, there exists $t_0>0$ such that $\|x(t_0)-x_0\|^2= \alpha+r$. Note that
$b_i^Tx(t_0)=b_i^T(x^*+t_0v) =b_i^Tx^* \le \beta_i+s_i$
 and \begin{eqnarray*}
F(x(t_0))&= & (x^*+t_0v)^T(A-\lambda_{\rm min}(A)I_n)(x^*+t_0v) \\
& & +(a+2\lambda_{\min}(A)x_0)^T(x^*+t_0v)+(\gamma-\lambda_{\rm min}(A)\|x_0\|^2) \\
&= & (x^*)^T(A+\lambda_{\rm min}(A)I_n)x^*+ (a+2\lambda_{\min}(A)x_0)^Tx^*+(\gamma-\lambda_{\rm min}(A)\|x_0\|^2)\\
& = & F(x^*).
 \end{eqnarray*}
This contradicts our assumption that any minimizer $x^*$ of $\min_{x \in \mathbb{R}^n}\{F(x): \|x-x_0\|^2 \le \alpha+r, b_i^Tx \le \beta_i+s_i\}$ satisfy $\|x^*-x_0\|^2<\alpha+r$.

Suppose that case 2 holds, i.e., $(a+2\lambda_{\min}(A)x_0)^Tv \neq 0$. By replacing $v$ with $-v$ if necessary, we may assume without loss of generality that
$(a+2\lambda_{\min}(A)^Tx_0)^Tv < 0.$
As $\|x^*-x_0\|^2 < \alpha+r$, there exists $t_0>0$ such that $\|x(t)-x_0\|^2 \le \alpha+r$ for all $t \in (0,t_0]$. Note that
$b_i^Tx(t_0)=b_i^T(x^*+t_0v) =b_i^Tx^* \le \beta_i+s_i$
 and \begin{eqnarray*}
F(x(t_0))&= & (x^*+t_0v)^T(A-\lambda_{\rm min}(A)I_n)(x^*+t_0v) \\
& & +(a+2\lambda_{\min}(A)x_0)^T(x^*+t_0v)+(\gamma-\lambda_{\rm min}(A)\|x_0\|^2) \\
& < & (x^*)^T(A-\lambda_{\rm min}(A)I_n)x^*+ (a+2\lambda_{\min}(A)x_0)^Tx^*+(\gamma-\lambda_{\rm min}(A)\|x_0\|^2)\\
& = & F(x^*).
 \end{eqnarray*}
This contradicts our assumption that $x^*$ is a minimizer.
\end{proof}
As a consequence, we deduce the hidden convexity of the well-known trust region system.
\begin{corollary}{\rm (Polyak \cite[Theorem 2.2]{Polyak})}
Let $f(x)=x^TAx+a^Tx+\gamma$ and $g_0(x)=\|x-x_0\|^2-\alpha$ where
$A \in S^{n \times n}$, $a,x_0 \in \mathbb{R}^n$ and $\gamma,\alpha \in \mathbb{R}$. Then, $U(f,g_0)$ is convex.
\end{corollary}
\begin{proof}
Let $b_i=0$, $i=1,\ldots,m$ $($so, ${\rm dim}\, {\rm span}\{b_1,\ldots,b_m\}=0$ $)$.
Then the dimension condition (\ref{eq:DC}) reduces to
${\rm dim}{\rm Ker}(A -\lambda_{\rm min}(A)I_n) \ge 1$ which is always satisfied.
So, Theorem \ref{prop:2.2} shows the  $U(f,g_0)$ is always convex.
\end{proof}

\begin{remark}{\bf (Observations on the Dimension Condition)} {\rm
We observe that the dimension condition (\ref{eq:DC}) in the case of quadratic programs with one linear inequality constraint, i.e. $m=1$ in (\ref{eq:DC}), has been used to establish strong duality in \cite{beck-eldar}. This conclusion is deduced in Corollary 4.2 of Section 4.

Moreover, a close inspection of the proof of Theorem \ref{prop:2.2} suggests that,
for a general system of quadratic functions $f(x)=x^TAx+a^Tx+\gamma$, $g_0(x)=\|x-x_0\|^2-\alpha$ and $g_i(x)=\|Bx\|^2+b_i^Tx-\beta_i$, $i=1,\ldots,m$ with
 $B \in \mathbb{R}^{l \times n}$ for some $l \in \mathbb{N}$, $U(f,g_0,g_1,\ldots,g_m)$ can be shown to be convex under a modified dimension condition.
 %``$ {\rm dim}\big({\rm Ker}(A -\lambda_{\rm min}(A)I_n) \cap {\rm Ker}(Q) \big) \ge s +1,
%$''.
This will be given later in Section 6.}
\end{remark}
\setcounter{equation}{0}
\section{Exact SDP Relaxations}
In this section, we establish that a semi-definite relaxation of the model problem (P)
\begin{eqnarray*}
(P) & \min_{x \in \mathbb{R}^n}  &  x^TAx+a^Tx \\
& \mbox{ s.t. } & \|x-x_0\|^2 \le \alpha, \\
& & b_i^Tx \le \beta_i, \, i=1,\ldots,m,
\end{eqnarray*}
is exact under the dimension condition. Importantly, it holds without the Slater condition. To formulate a SDP relaxation of (P), let us introduce the following $(n+1) \times (n+1)$ matrices: $M= \left(\begin{array}{cc}
A & a/2 \\
a^T/2 & 0
           \end{array}\right),$
\begin{equation}\label{eq:Hi}
 H_0=\left(\begin{array}{cc}
I_n & -x_0 \\
-x_0^T & \|x_0\|^2-\alpha
           \end{array}\right) \mbox{ and } H_i=\left(\begin{array}{cc}
0 & b_i/2 \\
b_i^T/2 & -\beta_i
           \end{array}\right), i=1,\ldots,m.
\end{equation}
Note that $x^TAx+a^Tx={\rm Tr}(MX)$, $\|x-x_0\|^2-\alpha={\rm Tr}(H_0X)$ and $b_i^Tx-\beta_i={\rm Tr}(H_iX)$ where
$X=\tilde{x}\tilde{x}^T$ with $\tilde{x}=(x^T,1)^T$. Thus, the model problem can be equivalently rewritten as
\begin{eqnarray*}
& \min_{X \in S^{n+1}_+} & {\rm Tr}(MX)  \\
& \mbox{ s.t. } &  {\rm Tr}(H_0 X) \le 0, \\
& & {\rm Tr}(H_i X) \le 0, i=1,\ldots,m \\
& & X_{n+1,n+1}=1, {\rm rank}(X)=1,
\end{eqnarray*}
where ${\rm rank}(X)$ denotes the rank of the matrix $X$ and $X_{n+1,n+1}$ is the element of $X$ that lies in the $n+1^{\rm th}$ row and $n+1^{\rm th}$ column. By removing the rank one constraint, we obtain the following  semi-definite relaxation of (P)
\begin{eqnarray*}
(SDRP) & \min_{X \in S^{n+1}_+} & {\rm Tr}(MX)  \\
& \mbox{ s.t. } &  {\rm Tr}(H_0 X) \le 0, \\
& & {\rm Tr}(H_i X) \le 0, i=1,\ldots,m \\
& & X_{n+1,n+1}=1.
\end{eqnarray*}
The semi-definite relaxation problem (SDRP) is a convex program over a matrix space. Its convex dual problem can be stated as follows
\begin{eqnarray*}
(D) & &\displaystyle  \max_{\mu \in \mathbb{R}, \;\lambda_i \ge 0, i=0,\ldots, m} \{ \mu : M+\sum_{i=0}^m \lambda_i H_i \succeq \left( \begin{array}{cc}
0& 0 \\
 0 & \mu
 \end{array}
\right)\} \\
& = & \max_{\lambda_i \ge 0, \atop i=0,\ldots,m} \min_{x \in \mathbb{R}^n}\{x^TAx+a^Tx+ \lambda_0(\|x-x_0\|^2-\alpha)+\sum_{i=1}^m\lambda_i (b_i^Tx-\beta_i)\},
\end{eqnarray*}
which coincides with the Lagrangian dual problem of (P). Clearly, (SDRP) and (D) are semi-definite linear programming problems and hence can be solved efficiently, whereas the original problem (P) which is a non-convex quadratic program with multiple constraints, is, in general, a computationally hard problem. Therefore, it is of interest to study when the semi-definite relaxation is exact in the sense that $\min(P)=\min(SDRP)$. For related most recent results on exact SDP relaxations, see \cite{jl-orl}.

If $A$ is positive semidefinite, then the  problem (P) is a convex quadratic optimization problem which is known to enjoy nice properties such as strong duality and exact relaxation. {\it Therefore, from now on, we assume that $A$ is not positive semidefinite and so, has at least one negative eigenvalue.}

\begin{theorem}\label{th:value} {\bf (Exact SDP-relaxation)} Suppose that the dimension condition (\ref{eq:DC}) is satisfied.
%\[
%{\rm dim}{\rm Ker}(A -\lambda_{\rm min}(A)I_n)+{\rm dim}(\bigcap_{i=1}^mb_i^{\bot}) \ge n +1 .
%\]
Then, the semi-definite relaxation is exact, i.e., $\min(P) = \min(SDRP)$.
\end{theorem}

\begin{proof}
${\bf [\min(P)=\max(D)<+\infty]}$. We first prove that there is no duality gap between (P) and (D) under
the dimension condition. It is known that this will follow if we show that the optimal value function of (P)
\[
v(s_0,s_1,\ldots,s_m):=\inf_{x \in \mathbb{R}^n}\{x^TAx+a^Tx: \|x-x_0\|^2 \le \alpha+s_0, \, b_i^Tx \le \beta_i+s_i, \, i=1,\ldots,m\},
\]
is lower semicontinuous and convex function on $\mathbb{R}^{m+1}$ (See, for instance \cite{jw} for details).
To see this, we first note that
${\rm epi}v=U(f,g_0,g_1,\ldots,g_m)$
where $f(x)=x^TAx+a^Tx$, $g_0(x)=\|x-x_0\|^2-\alpha$ and $g_i(x)=b_i^Tx-\beta_i$, $i=1,\ldots,m$.
So, by Proposition \ref{prop:2.2}, ${\rm epi}v$ is a convex set, and so, $v$ is a convex function.
The lower semicontinuity of $v$ will follow from Proposition \ref{prop:2.1} as $U(f,g_0,g_1,\ldots,g_m)$ is a closed set.
\medskip

\noindent${\bf [\min(P)=\min(SDRP)]}$. By the construction of the SDP relaxation problem (SDRP) and the dual (D), it is easy see that
\[
\min(P) \ge \min(SDRP) \ge \max(D).
\]
As there is no duality gap between (P) and (D), we obtain that
$\min(P)=\min(SDRP)$.
\medskip

\noindent{\bf [Attainment of Minimum of (SDRP)]} We now show that the minimum in (SDRP) is attained. To see this, we only need to show the feasible set of (SDRP) is bounded. If not, then there exist $X^k \in S_{+}^{n+1}$ with
\[
X^k= \left(\begin{array}{cc}
Y^k & y^k \\
y^k & 1
\end{array}
 \right)
\]
such that $\|X^k\|_F:=\sqrt{{\rm Tr}(X^kX^k)} \rightarrow +\infty$, ${\rm Tr}(H_iX^k) \le 0$, $i=0,1,\ldots,m$ where $H_i$, $i=0,1,\ldots,m$ is defined as in (\ref{eq:Hi}). This implies that
\[
0 \le {\rm Tr}(Y^k) \le -\|x_0\|^2+\alpha+2(y^k)^Tx_0 \mbox{ and } b_i^Ty^k \le \beta_i.
\]
As $X^k \succeq 0$, we have $Y^k-y^k(y^k)^T \succeq 0$. So, $$\|y^k\|^2 ={\rm Tr}\big(y^k(y^k)^T\big) \le {\rm Tr}(Y^k) \le -\|x_0\|^2+\alpha+2(y^k)^Tx_0.$$ So, $y^k$ is bounded, and so ${\rm Tr}(Y^k)$ is also a bounded sequence.  Thus, both $Y^k$ and $y^k$ are bounded. It follows that $X_k$ is bounded which contradicts the fact that $\|X^k\|_F \rightarrow +\infty$.
\end{proof}

It should be noted that convexity of the set $U(f,g_0,g_1,\ldots,g_m)$ plays an important role in establishing
the exact SDP relaxation of (P). However, as we see in the following example, the convexity does not imply that problem (P) is equivalent to a convex optimization problem in the sense that they have the same minimizers.
\begin{example}
Consider $f(x)=x^2$, $g_0(x)=x^2-1$ and $g_1(x)=-x^2+1$. It can be checked that $$U(f,g_0,g_1)=\{(x^2,x^2-1,-x^2+1):x \in \mathbb{R}\}=\{(z,z-1,-z+1):z \ge 0\},$$
which is a closed and convex set. On the other hand, the corresponding optimization problem $\min_{x \in \mathbb{R}}\{x^2: x^2-1 \le 0, -x^2+1 \le 0\}$ cannot be equivalent to a convex optimization
problem as its solution set is $\{-1,1\}$ which is not a convex set.
\end{example}

One interesting feature of our SDP relaxation result is that its exactness is independent of the Slater condition. The following example illustrates that our SDP relaxation may be exact while the Slater condition fails.

\begin{example}{\bf (Exact SDP-relaxation without the Slater condition)}\label{ex:3.1}
Consider the three dimensional quadratic optimization problem with two linear inequalities:
 \begin{eqnarray*}
(EP) & \displaystyle \min_{(x_1,x_2,x_3) \in \mathbb{R}^3}  &  -x_1^2-x_2^2-x_3^2+3x_1+2x_2+2x_3 \\
& \mbox{ s.t. } & (x_1-1)^2+x_2^2 +x_3^2\le 1, \\
& &  x_1 \le 0, \\
& & x_1+x_2+x_3 \le 0.
\end{eqnarray*}
This can be written as  our model problem where $A=\left(\begin{array}{ccc}
-1 & 0 & 0 \\
0 & -1 & 0 \\
0 & 0 & -1
                                  \end{array}\right)
$, $a=(3,2,2)^T$, $x_0=(1,0,0)^T$, $\alpha=1$, $b_1=(1,0,0)^T$, $b_2=(1,1,1)$ and $\beta_1=\beta_2=0$.
Clearly, the only feasible point is $(0,0,0)$ and so, $\min(EP)=0$. We also note that the Slater condition fails. Let $s={\rm dim}\, {\rm span}\{b_1,b_2\}=2$. We see that
$$ {\rm dim}{\rm Ker}(A -\lambda_{\rm min}(A)I_n)=3=s+1.$$ So, the dimension condition is satisfied.

On the other hand, the SDP-relaxation of (EP) is given by
\begin{eqnarray*}
(SDRP_E) & \displaystyle \min_{X \in S^{4}} & -z_1+3z_4-z_5+2z_7-z_8+2z_9  \\
& \mbox{ s.t. } &  z_1-2z_4+z_5+z_8 \le 0  \\
& &  z_4 \le 0 \\
& & z_4+z_7+z_9 \le 0 \\
& & X=\left(\begin{array}{cccc}
z_1 & z_2 & z_3 & z_4 \\
z_2 & z_5 & z_6 & z_7 \\
z_3 & z_6 & z_8 & z_9 \\
z_4 & z_7 & z_9 & 1
            \end{array}
 \right)\succeq 0.
\end{eqnarray*}
Since $z_1=z_2=\ldots=z_9=0$ is feasible for $(SDRP_E)$, $\min(SDRP_E)\le 0$. Moreover, for each feasible $X=\left(\begin{array}{cccc}
z_1 & z_2 & z_3 & z_4 \\
z_2 & z_5 & z_6 & z_7 \\
z_3 & z_6 & z_8 & z_9 \\
z_4 & z_7 & z_9 & 1
            \end{array}
 \right)\succeq 0$, we have $z_1 \ge 0$, $z_5 \ge 0$,
 \begin{equation} \label{eq:quad_ineq}
z_8 \ge z_9^2 \ge 0 \mbox{ and } z_5 \ge z_7^2 \ge 0.
 \end{equation}
 This gives us that
\[
 -2z_4 \le z_1-2z_4+z_5+z_8 \le 0
\]
and so, $z_4 \ge 0$. As $z_4 \le 0$, we have $z_4=0$ and so, $z_1+z_5+z_8 \le 0$. Hence, $z_1=z_5=z_8=0$ and $z_7=z_9=0$ (by (\ref{eq:quad_ineq})). Thus, $\min(SDRP_E)= 0=\min(EP)$.
\end{example}

% \begin{remark}  {\rm {\bf (Discussion on the dimension condition)}
% %If $\lambda_{\rm min}(A) \ge 0$, i.e., $A$ is positive semi-definite, then $\chi(\lambda_{\rm min}(A))=1$, and so, the dimension assumption is always satisfied as
% %\begin{eqnarray*}
% %& & {\rm dim}{\rm Ker}(A -\lambda_{\rm min}(A)I_n)+m \chi(\lambda_{\rm min}(A))+{\rm dim}(\bigcap_{i=1}^mb_i^{\bot}) \\
% %&= & {\rm dim}{\rm Ker}(A -\lambda_{\rm min}(A)I_n)+m +{\rm dim}(\bigcap_{i=1}^mb_i^{\bot})  \\
% %& \ge &  1+m + \max\{(n-m),0\} \ge n+1.
% %\end{eqnarray*}
% %Moreover, if $b_i=0$, $i=1,\ldots,m$ $($so, ${\rm dim}(\bigcap_{i=1}^mb_i^{\bot})=n$ $)$, then the dimension assumption is also always satisfied as
% %\begin{eqnarray*}
% %& & {\rm dim}{\rm Ker}(A -\lambda_{\rm min}(A)I_n)+m \chi(\lambda_{\rm min}(A))+{\rm dim}(\bigcap_{i=1}^mb_i^{\bot}) \\
% %&\ge & {\rm dim}{\rm Ker}(A -\lambda_{\rm min}(A)I_n)+n  \ge    n+1.
% %\end{eqnarray*}
% Denote the dimension of ${\rm span}\{b_1,\ldots,b_m\}$ by $s$ $(s \le n)$. Then  ${\rm dim}(\bigcap_{i=1}^mb_i^{\bot})=n-s$. So, in this case, the dimension assumption collapses to
% \[
% {\rm dim}{\rm Ker}(A -\lambda_{\rm min}(A)I_n) \ge s+1,
% \]
% which means that .
% }\end{remark}

In the following, we use a simple one-dimensional quadratic optimization problem to show that the SDP relaxation may not be exact if our sufficient dimension condition (\ref{eq:DC}) is not satisfied.
\begin{example}{\bf (Importance of sufficient dimension condition)}\label{ex:dimension}
Consider the minimization problem $$(EP_1) \ \min_{x \in \mathbb{R}}\{f(x): g_0(x) \le 0, g_1(x) \le 0\},$$
where $f(x)=x-x^2$, $g_0(x)=x^2-1$, $g_1(x)=-x$, $n=1$ and $m=1$. Then, $f(x)=x^TAx+a^Tx+r$ with $A=-1$, $a=1$ and $r=0$, $g_0(x)=\|x-x_0\|^2-\alpha$ with $x_0=0$, $\alpha=1$ and $g_1(x)=b_1^Tx-\beta_1$  with $b_1=1$ and  $\beta_1=0$. Clearly, ${\rm dim}{\rm Ker}(A -\lambda_{\rm min}(A)I_n)=1<2={\rm dim}\, {\rm span}\{b_1\}+1.$

The SDP relaxation of $(EP_1)$ is given by
\begin{eqnarray*}
(SDRP_{E1}) & \displaystyle \min_{X \in S^{2}} & -z_1+z_2  \\
& \mbox{ s.t. } &  z_1-1 \le 0  \\
& &  -z_2 \le 0 \\
& & X=\left(\begin{array}{cc}
z_1 & z_2  \\
z_2 & 1
            \end{array}
 \right)\succeq 0.
\end{eqnarray*}
It can be easily verified that $\min(EP_1)=0$ and $\min(SDRP_{E1})=-1$. Thus, the SDP relaxation of $(EP_1)$ is not exact.
\end{example}

Consider the quadratic optimization problem with one norm constraint and a rank-one quadratic inequality constraint:
\begin{eqnarray*}
(P_0) & \min_{x \in \mathbb{R}^n}  &  x^TAx+a^Tx \\
& \mbox{ s.t. } & \|x-x_0\|^2 \le \alpha, \\
& &  (b^Tx)^2 \le r,
\end{eqnarray*}
where $A \in S^{n \times n}$, $a,x_0,b \in \mathbb{R}^n$, $\alpha \in \mathbb{R}$ and $r \ge 0$.

Model problems of this form arise from the application of the trust-region method for the minimization of a nonlinear function with a discrete constraint. For instance, consider the trust-region approximation problem
\begin{eqnarray*}
& \min_{x \in \mathbb{R}^n}  &  x^TAx+a^Tx \\
& \mbox{ s.t. } & \|x-x_0\|^2 \le \alpha, \\
& &  b^Tx \in \{1, -1\}.
\end{eqnarray*}
The continuous relaxation of this problem becomes
\begin{eqnarray*}
& \min_{x \in \mathbb{R}^n}  &  x^TAx+a^Tx \\
& \mbox{ s.t. } & \|x-x_0\|^2 \le \alpha, \\
& &  -1\le b^Tx \le 1,
\end{eqnarray*}
which is, in turn, equivalent to $(P_0)$ with $r=1$.

The SDP-relaxation of $(P_0)$ is given by
\begin{eqnarray*}
(SDRP_0) & \min_{X \in S^{n+1}_+} & {\rm Tr}(\tilde{M}X)  \\
& \mbox{ s.t. } &  {\rm Tr}(\tilde{H}_0 X) \le 0 \\
& & {\rm Tr}(\tilde{H}_i X) \le 0, i=1,2 \\
& & X_{n+1,n+1}=1.
\end{eqnarray*}
where
\[
\tilde{M}= \left(\begin{array}{cc}
A & a/2 \\
a^T/2 & 0
           \end{array}\right), \tilde{H}_0=\left(\begin{array}{cc}
I_n & -x_0 \\
-x_0 & \|x_0\|^2-\alpha
           \end{array}\right)
\]
\[
 H_1=\left(\begin{array}{cc}
0 & b/2 \\
b^T/2 & -\sqrt{r}
           \end{array}\right) \mbox{ and }
 H_2=\left(\begin{array}{cc}
0 & -b/2 \\
-b^T/2 & \sqrt{r}
           \end{array}\right).
\]

We now obtain the following exact SDP-relaxation result for the problem $(P_0)$ under a dimension condition.

 \begin{corollary}{\bf (Trust-region model with rank-one constraint)}
% Suppose that the dimension condition (\ref{eq:DC}) is satisfied.
Suppose that
 ${\rm dim}{\rm Ker}(A -\lambda_{\rm min}(A)I_n) \ge 2.$
Then, the semi-definite relaxation is exact for $(P_0)$, i.e., $\min(P_0) = \min(SDRP_0)$.
\end{corollary}
\begin{proof} Note that $(b^Tx)^2 \le r$ is equivalent to $-\sqrt{r} \le b^Tx \le \sqrt{r}$. In this case the dimension condition of Theorem \ref{th:value} reduces to the assumption that
\[
 {\rm dim}{\rm Ker}(A -\lambda_{\rm min}(A)I_n) \ge {\rm dim}\, {\rm span}\{b,-b\}+1.
 \]
The conclusion follows from  Theorem \ref{th:value} and the fact that ${\rm dim}\, {\rm span}\{b,-b\} \le 1$.
 \end{proof}

\begin{remark} {\bf (Approximate S-lemma and SDP Relaxations)}{\rm
For a general homogeneous quadratic optimization problem with multiple convex quadratic constraints, an estimate for the ratio between the optimal value of the underlying quadratic optimization and its associated SDP relaxation problem has been given in \cite{BNR} (see Appendix). This result is known as an approximate S-lemma as it provides the approximate ratio from the SDP relaxation to the underlying problem. Clearly, Theorem \ref{th:value} shows that the ratio between the optimal value of the underlying quadratic optimization problem and its associated SDP relaxation problem  is one for the extended trust region problem (P), under the dimension condition. For other nonconvex quadratic optimization problems where the corresponding ratio also equals one, see \cite{Ye_Zhang}.

Consider the quadratic optimization problem with the
constraint set described by the intersection of an Euclidean ball and a box:
\begin{eqnarray*}
(P_1) & \min_{x \in \mathbb{R}^n}  &  x^TAx+a^Tx \\
& \mbox{ s.t. } & \|x\|^2 \le 1, \\
& &  -l_i \le x_i \le l_i, i=1,\ldots,n,
\end{eqnarray*}
where  $l_i > 0$. This class of nonconvex quadratic problems is known to be NP-hard. Indeed,  when $l_i \le \frac{1}{2}$ and $A$ is negative definite, the norm constraint,  $\|x\|^2 \le 1$, becomes superfluous and so, the problem $(P_1)$
reduces to the  quadratic concave minimization problem with bounded constraints which is an NP-hard problem (cf. \cite{BNR}). Using the approximate S-lemma of \cite{BNR} and a semidefinite programming relaxation, one can find an estimate for the value of the nonconvex quadratic problem $(P_1)$.
\medskip

We note that our dimension condition
fails for $(P_1)$. To see this, take $m=2n$, $b_i=e_i$, $i=1,\ldots,n$ and $b_i=-e_i$, $i=n+1,\ldots,2n$. Then, we see that ${\rm dim}\, {\rm span}\{b_1,\ldots,b_n\}=n$ in this case,
and so, the dimension condition reduces
to  ${\rm dim}{\rm Ker}(A -\lambda_{\rm min}(A)I_n)\ge n +1$ which is impossible.
\medskip

On the other hand, consider the following semi-definite relaxation of $(P_1)$ (see \cite{BNR})
\begin{eqnarray*}
(SDRP_1) & \min_{X \in S^{n+1}_+} & {\rm Tr}(MX)  \\
& \mbox{ s.t. } &  {\rm Tr}(H_0 X) \le 1, \\
& & {\rm Tr}(H_i X) \le 1, i=1,\ldots,n \\
& & {\rm Tr}(H_{2n+1}X) \le 1,
\end{eqnarray*}
where
\begin{equation}\label{eq:0099}
 {M}= \left(\begin{array}{cc}
A & a/2 \\
a^T/2 & 0
           \end{array}\right), {H}_0=\left(\begin{array}{cc}
I_n & 0 \\
0 & 0
           \end{array}\right)
\end{equation}
\begin{equation}\label{eq:0098}
 H_i=\left(\begin{array}{cc}
\frac{1}{l_i^2}{\rm diag}(e_i) & 0 \\
0 & 0
           \end{array}\right), \ i=1,\ldots,n ,
\mbox{ and }
  H_{n+1}=\left(\begin{array}{cc}
0_{n \times n} & 0 \\
0 & 1
           \end{array}\right).
\end{equation}

Following \cite{BNR}, we can get
\[
2\log (6n+6) \min(P_1) \le \min(SDRP_1) \le \min(P_1)  \le 0.
\]

To see this, we first note that  $\min(P_1)$ equals the optimal value of the following optimization problem
\begin{eqnarray*}
 & \min_{(x,t) \in \mathbb{R}^n \times \mathbb{R}}  &  x^TAx+ta^Tx  \\
& \mbox{ s.t. } & \|x\|^2 \le 1, \\
& &  \frac{1}{l_i^2}x_i^2 \le 1, i=1,\ldots,n, \\
& & t^2 \le 1.
\end{eqnarray*}
which is, in turn, equal to the negative of the optimal value of the following quadratically constrained quadratic problem
\[
 (QCQ_1) \max_{y=(x^T,t)^T \in \mathbb{R}^n \times \mathbb{R}}\{-y^TMy: y^TH_0y \le 1, y^TH_iy \le 1, i=1,\ldots,n+1\}
\]
where $M$ and $H_i$, $i=0,1,\ldots,n+1$ are defined as in (\ref{eq:0099}) and (\ref{eq:0098}).
Note that ${\rm rank}H_i=1$, $i=1,2,\ldots,n+1$ and $\sum_{i=0}^{n+1}H_i \succ 0$. So, {\rm \cite[Lemma A.6, Approximate S-lemma]{BNR}}  implies that
\begin{eqnarray} \label{eq:approximate}
 \max(QCQ_1) \le \min(SDP_1) &\le & 2\log(6 \sum_{i=1}^{n+1}{\rm rank}H_i)\max(QCQ_1) \nonumber \\
&= & 2\log (6n+6) \max(QCQ_1),
\end{eqnarray}
where $(SDP_1)$ is given by
\[
(SDP_1) \ \ \ \min_{\mu_0,\ldots,\mu_{n+1} \ge 0}\{\sum_{i=0}^{n+1} \mu_i : M+\sum_{i=0}^{n+1}\mu_kH_k \succeq 0\}.
\]
It can be verified that $(SDP_1)$ is the Lagrange dual problem of the semi-definite problem
\begin{eqnarray*}
(SP_1) & \max_{X \in S^{n+1}_+} & {\rm Tr}(-MX)  \\
& \mbox{ s.t. } &  {\rm Tr}(H_0 X) \le 1, \\
& & {\rm Tr}(H_i X) \le 1, i=1,\ldots,n \\
& & {\rm Tr}(H_{n+1}X) \le 1,
\end{eqnarray*}
and Slater condition holds for $(SP_1)$. So, $\min(SDP_1)=\max(SP_1)$. Finally, the conclusion follows from (\ref{eq:approximate})
by noting that $\max(SP_1)=-\min(SDRP_1)$.
}\end{remark}

\setcounter{equation}{0}
\section{Global Optimality and Strong Duality}
In this section, we present a necessary and sufficient condition for global optimality of (P) and consequently, obtain strong duality between (P) and (D) whenever the dimension condition is satisfied and Slater's condition holds for (P). Related global optimality and duality results for nonconvex quadratic optimization can be found in \cite{jll,jhl_OE,jrwnes,Yuan_Peng}.

\begin{theorem} \label{th:global}  {\bf (Necessary and sufficient global optimality condition)} For (P), suppose that there exists $\overline{x} \in \mathbb{R}^n$ with $\|\overline{x}-x_0\|^2 <\alpha$ and $b_i^T\overline{x}<\beta_i$, $i=1,\ldots,m$, and that the dimension condition (\ref{eq:DC}) is satisfied.
%\[{\rm dim}{\rm Ker}(A -\lambda_{\rm min}(A)I_n)+{\rm dim}(\bigcap_{i=1}^mb_i^{\bot}) \ge n +1 .
%\]
Let $x^*$ be a feasible point of (P). Then, $x^*$ is a global minimizer of (P) if and only if there exists $(\lambda_0,\lambda_1,\ldots,\lambda_m) \in \mathbb{R}_+^{m+1}$ such that the following condition holds:
\begin{eqnarray*}
\left \{\begin{array}{ll}
2\big((A+\lambda_0 I_n)x^*\big)=-\big(a+2\lambda_0 (x^*-x_0)+\sum_{i=1}^m \lambda_i b_i\big), & \mbox{ (KKT Condition)} \\
\lambda_0(\|x^*-x_0\|^2-\alpha)=0 \mbox{ and } \lambda_i(b_i^Tx^*-\beta_i)=0, i=1,\ldots,m, & \mbox{ (Complementary Slackness)} \\
A+\lambda_0 I_n \succeq 0, & \mbox{ (Second Order Condition)}.
\end{array} \right.
\end{eqnarray*}
\end{theorem}
\begin{proof}
%From the preceding theorem, $x^*$ is a global minimizer of (P) is equivalent to the fact that:
%there exists there exists $(\lambda_0,\lambda_1,\ldots,\lambda_m) \in \mathbb{R}^{m+1}$ such that
%\[
%(x^*)^TAx^*+a^Tx^*=\min_{x \in \mathbb{R}^n}\{x^TAx+a^Tx+ \lambda_0(\|x-x_0\|^2-\alpha)+\sum_{i=1}^m\lambda_i (b_i^Tx-\beta_i)\}.
%\]

{\bf [Necessary condition for optimality]}. Let $x^*$ be a global minimizer of (P). Then, the following inequality system has no solution:
\[
\|{x}-x_0\|^2 \le \alpha, \ b_i^T{x} \le \beta_i, \, i=1,\ldots,m, \ x^TAx+a^Tx < (x^*)^TAx^*+a^Tx^*.
\]
In particular, letting $\gamma=-((x^*)^TAx^*+a^Tx^*)$, the following inequality system also has no solution:
\[
\|{x}-x_0\|^2 < \alpha, \ b_i^T{x} < \beta_i, \, i=1,\ldots,m, \ x^TAx+a^Tx+\gamma < 0.
\]
Then,
$0 \notin  {\rm int}U(f,g_0,g_1,\ldots,g_m),$
where $$U(f,g_0,g_1,\ldots,g_m):=\{(f(x),g_0(x),g_1(x),\ldots,g_m(x)):x \in \mathbb{R}^n\}+\mathbb{R}_+^{m+2}$$ is a convex set by proposition \ref{prop:2.2}. Moreover, as $f,g_i$ are all continuous, we see
that
\[
 \{(f(x),g_0(x),g_1(x),\ldots,g_m(x)):x \in \mathbb{R}^n\}+{\rm int}\mathbb{R}_+^{m+2}={\rm int}U(f,g_0,g_1,\ldots,g_m)
\]
is also convex.

Now, by the convex separation theorem, there exists $(\mu,\tilde{\lambda}_0,\tilde{\lambda}_1,\ldots,\tilde{\lambda}_m) \in \mathbb{R}^{m+2}_+ \backslash\{0\}$ such that, for all $x \in \mathbb{R}^n$,
\[
\mu(x^TAx+a^Tx+\gamma)+\tilde{\lambda}_0(\|x-x_0\|^2-\alpha)+\sum_{i=1}^m\tilde{\lambda}_i (b_i^Tx-\beta_i) \ge 0.
\]
By the strict feasibility condition, we see that $\mu \neq 0$. Thus, for all $x \in \mathbb{R}^n$
\[
x^TAx+a^Tx+\gamma+{\lambda}_0(\|x-x_0\|^2-\alpha)+\sum_{i=1}^m{\lambda}_i (b_i^Tx-\beta_i) \ge 0.
\]
where $\lambda_i=\frac{\tilde{\lambda}_i}{\mu}$, $i=0,1,\ldots,m.$ Letting $x=x^*$, we see that
\[
{\lambda}_0(\|x^*-x_0\|^2-\alpha)+\sum_{i=1}^m{\lambda}_i (b_i^Tx^*-\beta_i) \ge 0.
\]
As $x^*$ is feasible for (P), it follows that
\[
\lambda_0(\|x^*-x_0\|^2-\alpha)=0 \mbox{ and } \lambda_i(b_i^Tx^*-\beta_i)=0, i=1,\ldots,m.
\]
Let $h(x):=x^TAx+a^Tx+ \lambda_0(\|x-x_0\|^2-\alpha)+\sum_{i=1}^m\lambda_i (b_i^Tx-\beta_i)$. Then, we see that $x^*$ is a global minimizer of $h$, and so, $\nabla h(x^*)=0$ and $\nabla^2 h(x^*) \succeq 0$. That is to say,
\[
2(A+\lambda_0 I_n)x^*+\big(a+2\lambda_0 (x^*-x_0)+\sum_{i=1}^m \lambda_i b_i\big)=0 \mbox{ and } A+\lambda_0 I_n \succeq 0.
\]

{\bf [Sufficient condition for optimality]} Conversely, if the optimality condition holds, then we see that $h(x):=x^TAx+a^Tx+ \lambda_0(\|x-x_0\|^2-\alpha)+\sum_{i=1}^m\lambda_i (b_i^Tx-\beta_i)$ is convex with $\nabla h(x^*)=0$ and $\nabla^2 h(x^*) \succeq 0$. So, $x^*$ is a global minimizer of $h$, and hence,
for all feasible point $x \in \mathbb{R}^n$ of (P),
\begin{eqnarray*}
 x^TAx+a^Tx & \ge & x^TAx+a^Tx+\lambda_0(\|x-x_0\|^2-\alpha)+\sum_{i=1}^m\lambda_i (b_i^Tx-\beta_i)\\
  & \ge  &  (x^*)^TAx^*+a^Tx^*+ \lambda_0(\|x^*-x_0\|^2-\alpha)+\sum_{i=1}^m\lambda_i (b_i^Tx^*-\beta_i) \\
  &= & (x^*)^TAx^*+a^Tx^*,
\end{eqnarray*}
where the last equality follows by the complementary condition. Thus, $x^*$ is a global minimizer of (P).
\end{proof}

\medskip

Consider the Lagrangian dual problem of (P):
\begin{eqnarray*}
(D) & & \max_{\lambda_i \ge 0} \min_{x \in \mathbb{R}^n}\{x^TAx+a^Tx+ \lambda_0(\|x-x_0\|^2-\alpha)+\sum_{i=1}^m\lambda_i (b_i^Tx-\beta_i)\}.
\end{eqnarray*}
We now show that the strong duality holds under the dimension condition together with the Slater condition.
\begin{corollary}\label{th:strong_duality} {\bf (Strong Duality)}
Suppose that there exists $\overline{x} \in \mathbb{R}^n$ with $\|\overline{x}-x_0\|^2 <\alpha$ and $b_i^T\overline{x}<\beta_i$, $i=1,\ldots,m$, and that the dimension condition (\ref{eq:DC}) is satisfied.
%\[{\rm dim}{\rm Ker}(A -\lambda_{\rm min}(A)I_n)+m \chi(\lambda_{\rm min}(A))+{\rm dim}(\bigcap_{i=1}^mb_i^{\bot}) \ge n +1 .
%\]
Then, strong duality holds, i.e.,
\begin{eqnarray}\label{eq:maximum}
& & \min_{x \in \mathbb{R}^n}\{x^TAx+a^Tx:\|x-x_0\|^2 \le \alpha, \, b_i^Tx \le \beta_i, \, i=1,\ldots,m\} \nonumber \\
&= &\max_{\lambda_i \ge 0} \min_{x \in \mathbb{R}^n}\{x^TAx+a^Tx+ \lambda_0(\|x-x_0\|^2-\alpha)+\sum_{i=1}^m\lambda_i (b_i^Tx-\beta_i)\}.
\end{eqnarray}
where the maximum in (\ref{eq:maximum}) is attained.
\end{corollary}
\begin{proof} First of all, we note that the following weak duality always holds:
\begin{eqnarray*}
& & \min_{x \in \mathbb{R}^n}\{x^TAx+a^Tx:\|x-x_0\|^2 \le \alpha, \, b_i^Tx \le \beta_i, \, i=1,\ldots,m\}\\
& \ge &\max_{\lambda_i \ge 0} \min_{x \in \mathbb{R}^n}\{x^TAx+a^Tx+ \lambda_0(\|x-x_0\|^2-\alpha)+\sum_{i=1}^m\lambda_i (b_i^Tx-\beta_i)\}.
\end{eqnarray*}
To see the reverse inequality, let $x^*$ be a minimizer of $\min_{x \in \mathbb{R}^n}\{x^TAx+a^Tx:\|x-x_0\|^2 \le \alpha, \, b_i^Tx \le \beta_i, \, i=1,\ldots,m\}$. Then, by Theorem \ref{th:global},
there exists $(\lambda_0,\lambda_1,\ldots,\lambda_m) \in \mathbb{R}_+^{m+1}$ such that the following condition holds:
\begin{eqnarray*}
\left \{\begin{array}{ll}
2(A+\lambda_0 I_n)x^*=-\big(a+2\lambda_0 (x^*-x_0)+\sum_{i=1}^m \lambda_i b_i\big), &  \\
\lambda_0(\|x^*-x_0\|^2-\alpha)=0 \mbox{ and } \lambda_i(b_i^Tx^*-\beta_i)=0, i=1,\ldots,m, &  \\
A+\lambda_0 I_n \succeq 0 & .
\end{array} \right.
\end{eqnarray*}
Then we see that $h(x):=x^TAx+a^Tx+ \lambda_0(\|x-x_0\|^2-\alpha)+\sum_{i=1}^m\lambda_i (b_i^Tx-\beta_i)$ is convex with $\nabla h(x^*)=0$ and $\nabla^2 h(x^*) \succeq 0$. So, $x^*$ is a global minimizer of $h$, and hence, for all $x \in \mathbb{R}^n$
\begin{eqnarray*}
 & & x^TAx+a^Tx+ \lambda_0(\|x-x_0\|^2-\alpha)+\sum_{i=1}^m\lambda_i (b_i^Tx-\beta_i) \\
& \ge &  (x^*)^TAx^*+a^Tx^*+ \lambda_0(\|x^*-x_0\|^2-\alpha)+\sum_{i=1}^m\lambda_i (b_i^Tx^*-\beta_i) \\
&= &  (x^*)^TAx^*+a^Tx^*.
\end{eqnarray*}
Thus, the reverse inequality is true and the maximum in (\ref{eq:maximum}) is attained. So, the conclusion follows.
\end{proof}

It is easy to see that, for the extended trust-region model problem with linear inequality constraints, our Corollary \ref{th:strong_duality} shows that the ratio between the optimal value of the underlying problem and its associated SDP relaxation problem is one whenever the dimension condition is satisfied. For other quadratic optimization problems where the approximate ratio, is one see \cite{Ye_Zhang}.

Consider the following nonconvex quadratic optimization problem
subject to a norm constraint and a linear  constraint:
\begin{eqnarray*}
(P_2) & \min_{x \in \mathbb{R}^n}  &  x^TAx+a^Tx \\
& \mbox{ s.t. } & \|x-x_0\|^2 \le \alpha, \\
& & b_1^Tx \le \beta_1,
\end{eqnarray*}
where $b_1 \in \mathbb{R}^n$ and  $\beta_1 \in \mathbb{R}$.

As a corollary of Theorem \ref{th:strong_duality}, we now establish strong duality for $(P_1)$ which was established in \cite{beck-eldar}.
\begin{corollary}\label{cor:1}{\bf (Trust-region model with single linear constraint) \cite[Theorem 3.6]{beck-eldar}}
%For $(P_1)$, suppose that there exists $\overline{x} \in \mathbb{R}^n$ with $\|\overline{x}-x_0\|^2 <\alpha$ and $b_1^T\overline{x}<\beta_1$, and
%${\rm dim}{\rm Ker}(A -\lambda_{\rm min}(A)I_n) \ge 2$.
 For problem $(P_1)$, suppose that ${\rm dim}\big({\rm Ker}(A -\lambda_{\rm min}(A)I_n)\big) \ge 2$ and suppose
 that there exists $\overline{x}$ such that
$\|\overline{x}-x_0\|^2 < \alpha$ and $b_1^T\overline{x} < \beta_1$. Then, strong duality holds for problem $(P_1)$, i.e.,
\begin{eqnarray}\label{eq:maximum1}
& & \min_{x \in \mathbb{R}^n}\{x^TAx+a^Tx:\|x-x_0\|^2 \le \alpha, \, b_1^Tx \le \beta_1\} \nonumber \\
&= &\max_{\lambda_0,\lambda_1 \ge 0} \min_{x \in \mathbb{R}^n}\{x^TAx+a^Tx+ \lambda_0(\|x-x_0\|^2-\alpha)+\lambda_1 (b_1^Tx \le \beta_1)\},
\end{eqnarray}
and the maximum in (\ref{eq:maximum1}) is attained.
\end{corollary}
\begin{proof} The conclusion follows by letting $l=1$ in Corollary \ref{th:strong_duality} and noting that $s={\rm span}\{b_1\} \le 1$.
\end{proof}

Let us note that, if the Slater condition is not satisfied, strong duality may  fail while the SDP relaxation is exact. Indeed, the same problem discussed in Example \ref{ex:3.1} can be used to illustrate this situation.
\begin{example}{\bf (Exact SDP-relaxation without Strong duality)}
Consider the same problem in Example \ref{ex:3.1}:
 \begin{eqnarray*}
(EP) & \displaystyle \min_{(x_1,x_2,x_3) \in \mathbb{R}^3}  &  -x_1^2-x_2^2-x_3^2+3x_1+2x_2+2x_3 \\
& \mbox{ s.t. } & (x_1-1)^2+x_2^2+x_3^2 \le 1, \\
& &  x_1 \le 0, \\
& & x_1+x_2+x_3 \le 0.
\end{eqnarray*}
 We have already shown that $\min(EP)=0$, the Slater condition fails for (EP) and the SDP relaxation of (EP) is exact. We now show that strong duality fails.  The Lagrangian dual problem of (EP) is
\begin{eqnarray*}
 & & \displaystyle \max_{\lambda_0,\lambda_1 \ge 0} \min_{(x_1,x_2,x_3) \in \mathbb{R}^3}\{-x_1^2-x_2^2-x_3^2+3x_1+2x_2+2x_3+\lambda_0\big((x_1-1)^2+x_2^2+x_3^2-1\big)\\
& & \ \ \ \ \ \ \ \ \ \ \ \ \ \ \ \ \ \ \ +\lambda_1 x_1+\lambda_2(x_1+x_2+x_3)\} \\
 & = & \max_{\lambda_0,\lambda_1 \ge 0} \min_{(x_1,x_2,x_3) \in \mathbb{R}^3}\{(\lambda_0-1)x_1^2+(\lambda_1+\lambda_2-2\lambda_0+3)x_1+(\lambda_0-1)x_2^2+(2+\lambda_2)x_2\\
& &  \ \ \ \ \ \ \ \ \ \ \ \ \ \ \ \ \ \ \  +(\lambda_0-1)x_3^2+(2+\lambda_2)x_3\}.
\end{eqnarray*}
%We now proceed by the method of contradiction and suppose that there exists $\overline{\lambda}_0,\overline{\lambda}_1 \ge 0$ such that
%\[
% \min_{(x_1,x_2) \in \mathbb{R}^2}\{ (\overline{\lambda}_0-1)x_1^2+(\overline{\lambda}_1-2\overline{\lambda}_0+3)x_1+(\overline{\lambda}_0-1)x_2^2+2x_2\}=\min(EP)=0.
%\]
For each $\lambda_0,\lambda_1 \ge 0$,
{\small
\begin{eqnarray*}
& & \min_{(x_1,x_2,x_3) \in \mathbb{R}^3}\{(\lambda_0-1)x_1^2+(\lambda_1-2\lambda_0+3)x_1
+(\lambda_0-1)x_2^2+(2+\lambda_2)x_2 +(\lambda_0-1)x_3^2+(2+\lambda_2)x_3\} \\
& =& \left\{\begin{array}{ccl}
        -\infty, & \mbox{ if } & \lambda_0<1, \\
       -\infty, & \mbox{ if } & \lambda_0=1,  \\
       <0, & \mbox{ if } & \lambda_0>1.
       \end{array}
 \right.
\end{eqnarray*}
}
Hence, strong duality fails.
\end{example}

As a consequence of our strong duality theorem, we derive a dual characterization for the non-negativity of a nonconvex quadratic function over the extended trust-region constraints. This characterization can be regarded as a form of the celebrated S-lemma \cite{bn}. See Appendix for variants of S-lemma.
\begin{corollary}{\bf (S-lemma for extended trust-regions)}\label{cor:S_lemma}
Let  $x_0,a,b_i \in \mathbb{R}^n$ and $\gamma,\beta_i,\alpha \in \mathbb{R}$, $i=1,\ldots,m$. Suppose that  there exists $\overline{x} \in \mathbb{R}^n$ with $\|\overline{x}-x_0\|^2 <\alpha$ and $b_i^T\overline{x}<\beta_i$, $i=1,\ldots,m$, and that the dimension condition (\ref{eq:DC}) is satisfied.
%\[{\rm dim}{\rm Ker}(A -\lambda_{\rm min}(A)I_n)+{\rm dim}(\bigcap_{i=1}^mb_i^{\bot}) \ge n +1 .
%\]
Then, the following statements are equivalent:
\begin{itemize}
\item[{\rm (1)}] $\|x-x_0\|^2-\alpha \le 0, \, b_i^Tx-\beta_i \le 0, \, i=1,\ldots,m \ \Rightarrow \ x^TAx+a^Tx+\gamma \ge 0.$
\item[{\rm (2)}] $(\exists \, \lambda_i \ge 0, i=0,1,\ldots,m)(\forall \, x \in \mathbb{R}^n)$ \ $$(x^TAx+a^Tx+\gamma)+\lambda_0(\|x-x_0\|^2-\alpha)+\sum_{i=1}^m\lambda_i (b_i^Tx-\beta_i) \ge 0.$$
\end{itemize}
\end{corollary}
\begin{proof}
We only need to show ${\rm (1)} \Rightarrow {\rm (2)}$ as the converse implication always holds. To see this, suppose {\rm (1)} holds. Then, the optimal value of the following optimization problem is greater than $-\gamma$
\[
\min_{x \in \mathbb{R}^n}\{x^TAx+a^Tx: \|x-x_0\|^2 \le \alpha , b_i^Tx \le \beta_i, i=1,\ldots,m\}.
\]
Then, Corollary \ref{th:strong_duality} implies that
\begin{eqnarray} \label{eq:maximum2}
%& & \min_{x \in \mathbb{R}^n}\{x^TAx+a^Tx: \|x-x_0\|^2 \le \alpha , b_i^Tx \le \beta_i, i=1,\ldots,m\}\nonumber \\
&  & \min_{x \in \mathbb{R}^n}\{x^TAx+a^Tx:\|x-x_0\|^2 \le \alpha, \, b_i^Tx \le \beta_i, \, i=1,\ldots,m\} \nonumber \\
&= &\max_{\lambda_i \ge 0} \min_{x \in \mathbb{R}^n}\{x^TAx+a^Tx+ \lambda_0(\|x-x_0\|^2-\alpha)+\sum_{i=1}^m\lambda_i (b_i^Tx-\beta_i)\} \ge -\gamma,
\end{eqnarray}
and the maximum in (\ref{eq:maximum2}) is attained. So, {\rm (2)} follows.
\end{proof}
Recall that the celebrated S-lemma states that, for two quadratic functions $f,g$, $[g(x) \le 0 \ \Rightarrow \ f(x) \ge 0]$ is equivalent to the existence of $\lambda \ge 0$ such that
$f+\lambda g$ is always nonnegative. Note that, in the case where $b_i=0$ and $\beta_i=1$,  the dimension condition is always satisfied as ${\rm dim}{\rm Ker}(A -\lambda_{\rm min}(A)I_n) \ge 1$ and ${\rm dim} \, {\rm span}\{b_1,\ldots,b_m\}=0$, and so, the above corollary reduces to the S-lemma in the case where $g=\|x-x_0\|^2-\alpha$.
\medskip

It is worth noting that in Corollary \ref{cor:S_lemma}, the strict feasibility condition cannot be dropped even if the dimension condition is satisfied.
 To see this, consider the following one-dimensional quadratic functions
$f(x)=x$ and $g_0(x)=x^2$. It can be verified that the dimension condition is satisfied and $[g_0(x) \le 0 \, \Rightarrow \, f(x) \ge 0]$. On the other hand, for any $\lambda \ge 0$,
\[
\inf_{x \in \mathbb{R}}\{f(x)+\lambda g(x)\}=\left\{\begin{array}{ccc}
 -\frac{1}{4\lambda}<0,                                                   & \mbox{ if } & \lambda > 0,\\
-\infty, &\mbox{ if } &  \lambda=0.
                                                    \end{array}
 \right.
\]
Therefore, Corollary \ref{cor:S_lemma} can fail if the strict feasibility condition is not satisfied.

On the other hand, if  the strict feasibility condition fails, we now show that a new form  of asymptotic S-lemma still holds. For related asymptotic S-lemma of this form for general quadratic constraint without Slater condition see \cite{jhl_OE}.
\begin{corollary}{\bf (Asymptotic S-lemma)}
Let $A \in S^n$, $x_0,a,b_i \in \mathbb{R}^n$ and $\gamma,\beta_i,\alpha \in \mathbb{R}$, $i=1,\ldots,m$ with $\{x:\|x-x_0\|^2 \le \alpha , \, b_i^Tx \le \beta_i, \, i=1,\ldots,m\}\neq \emptyset$. Suppose that  the dimension condition (\ref{eq:DC}) is satisfied.
%\[{\rm dim}{\rm Ker}(A -\lambda_{\rm min}(A)I_n)+{\rm dim}(\bigcap_{i=1}^mb_i^{\bot}) \ge n +1 .
%\]
Then, the following statements are equivalent:
\begin{itemize}
\item[{\rm (1)}] $\|x-x_0\|^2-\alpha \le 0, \, b_i^Tx-\beta_i \le 0, \, i=1,\ldots,m \ \Rightarrow \ x^TAx+a^Tx+\gamma \ge 0.$
\item[{\rm (2)}] $(\forall \, \epsilon>0)(\exists \, \lambda_i \ge 0, i=0,1,\ldots,m)(\forall \, x \in \mathbb{R}^n)$ \    $$(x^TAx+a^Tx+\gamma)+\lambda_0(\|x-x_0\|^2-\alpha)+\sum_{i=1}^m\lambda_i (b_i^Tx-\beta_i)+\epsilon \ge 0.$$
\end{itemize}
\end{corollary}
\begin{proof}
$[{\rm (1)} \Rightarrow {\rm (2)}]$  Suppose that {\rm (1)} holds. Let $f(x)=x^TAx+a^Tx+\gamma$, $g_0(x)=\|x-x_0\|^2-\alpha$ and $g_i(x)=b_i^Tx-\beta_i$, $i=1,\ldots,m$. Then, for each $\epsilon>0$, $(-\epsilon,0,0,\ldots,0) \notin U(f,g_0,g_1,\ldots,g_m)$. As the dimension condition (\ref{eq:DC}) holds, it follows from Proposition \ref{prop:2.1} and Theorem \ref{prop:2.2} that
$U(f,g_0,g_1,\ldots,g_m)$ is a closed convex set. So, the strong separation theorem gives us that $(\mu,\overline{\lambda}_0,\overline{\lambda}_1,\ldots,\overline{\lambda}_m) \in \mathbb{R}^{m+2}_+ \backslash\{0\}$ and $\delta \in \mathbb{R}$ such that
\[
-\mu \epsilon < \delta \le \mu f(x)+\sum_{i=0}^m\overline{\lambda}_i g_i(x) \mbox{ for all } x  \in \mathbb{R}^n.
\]
Then, $\mu >0$. Otherwise, $\mu=0$. Then, $\sum_{i=0}^m\lambda_i g_i(x) \ge \delta>0$ for all $x \in \mathbb{R}^n$. This is impossible as $\sum_{i=0}^m\lambda_i g_i(a) \le 0$ for all
 $a \in \{x:g_i(x) \le 0, i=0,1,\ldots,m\}$. So, {\rm (2)} follows with $\lambda_i=\frac{\overline{\lambda}_i}{\mu}$, $i=0,1,\ldots,m$.

$[{\rm (2)} \Rightarrow {\rm (1)}]$ For any $x$ with $g_i(x) \le 0$, then {\rm (2)} implies that for each $\epsilon>0$, there exist $\lambda_i \ge 0$ such that for all $x \in \mathbb{R}^n$,
\[
 0 \le f(x)+\sum_{i=0}^m\lambda_i g_i(x)+\epsilon \le f(x)+\epsilon.
\]
Letting $\epsilon \rightarrow 0$, we see that $f(x) \ge 0$, and so, {\rm (1)} follows.
\end{proof}
Before we end this section, let us use the preceding example to illustrate the new form of asymptotic S-lemma.
\begin{example}{\bf (Example illustrating the asymptotic S-lemma)}
Consider the following one-dimensional quadratic functions $f(x)=x$ and $g_0(x)=x^2$. It can be easily checked that $[x^2 \le 0 \Rightarrow x \ge 0]$ and the dimension condition is satisfied. Now, for each $\epsilon>0$, $x+\frac{1}{4\epsilon} x^2+\epsilon=(\frac{1}{2\sqrt{\epsilon}}x+\sqrt{\epsilon})^2 \ge 0$. So, our form of asymptotic S-lemma holds.
\end{example}

\setcounter{equation}{0}
 \section{Applications to Robust Optimization}
In this section, we establish SDP characterizations of the solution of a robust least squares problem (LSP) as well as a robust second order cone programming problem (SOCP) where the uncertainty set is  given by the intersection of the norm constraint and the polyhedral constraint. Consequently, we show that solving the robust (LSP) or a robust (SOCP) is equivalent to solving a semi-definite linear programming problem and so the solution can be validated in polynomial time.

Let us note first that, for a $(p \times q)$ matrix $M$, ${\rm vec}(M)$ denotes the vector in $\mathbb{R}^{pq}$ obtained
by stacking the columns of $M$. The tensor product of $I_n$ and a matrix
$M \in \mathbb{R}^{p \times p}$ is defined by
$$I_n \otimes M:=\left(\begin{array}{ccccc}
% c_0 & c_{n-1} & \ldots & c_{2} & c_{1} \\
% c_{1} & c_0 & c_{n-1} & \ddots & c_2 \\
% \vdots & \ddots & \ddots & \ddots & \vdots \\
% c_{n-2} &  & \ddots & c_0 & c_{n-1}\\
% c_{n-1} & c_{n-2} & \ldots & c_1 & c_0
M & 0 & 0 & 0 & 0 \\
0 & M & 0 & \ddots & 0 \\
\vdots & \ddots & \ddots & \ddots & \vdots \\
0 & & \ddots & M & 0 \\
0 & 0 & \ldots & 0 &M
          \end{array}\right) \in \mathbb{R}^{np \times np} .$$

          Consider the uncertainty set which is
 described by a matrix norm constraint and polyhedral constraints, i.e.,
\begin{equation}\label{eq:pp0}
\mathcal{U}=\{\tilde{A}^{(0)} + \Delta: \Delta \in \mathbb{R}^{k \times (n+1)}, \|\Delta-\overline{\Delta}\|_F \le \rho, \  (w^{j})^T{\rm vec}\Delta \le \beta^j, j=1,\ldots,l \},
\end{equation}
where $\tilde{A}^{(0)}:=(A^{(0)},a^{(0)}) \in \mathbb{R}^{k \times n} \times \mathbb{R}^k= \mathbb{R}^{k \times (n+1)}$ is the data of a given model, examined in this Section (see Sections 5.1 and 5.2),
 and $\|M\|_F$ is the Frobenius norm defined by $\|M\|_F=\sqrt{{\rm Tr}(M^TM)}$. In the special case when $l=2$, $w^2=-w^1$ and $\beta^1=-\beta^2=1$, this uncertainty set reduces to an intersection of two ellipsoids which was examined in \cite{Beck0}.

We say $(x,\lambda)\in \mathbb{R}^n \times \mathbb{R}$ is robust feasible for the quadratic constraint of the form $\|Ax-a\|^2 \le \lambda$ with respect to the uncertainty set $\mathcal{U}$ whenever $\max_{(A,a) \in \mathcal{U}} \|Ax-a\|^2 \le \lambda$. This form of quadratic constraint arises in a robust least squares models as well as a second order cone programming models.

We now show that checking robust feasibility is equivalent to solving a SDP, under suitable conditions.

\begin{lemma}{\bf (SDP reformulation of robust feasibility)} \label{lemma:5.1}
Let $(x,\lambda) \in \mathbb{R}^n \times \mathbb{R}$ and $\mathcal{U}$ be given as in (\ref{eq:pp0}). Suppose that $k \ge s+1$, where $k$ is the number of rows in the matrix data of $\mathcal{U}$ and $s={\rm dim}\, {\rm span}\{w^1,\ldots,w^l\}$, and that $\{\Delta: \|\Delta-\overline{\Delta}\|_F < \rho, \  (w^{j})^T{\rm vec}\Delta < \beta^j, j=1,\ldots,l\} \neq \emptyset.$.  Then, $(x,\lambda)$ is robust feasible for the quadratic constraint $\|Ax-a\|^2 \le \lambda$ with respect to the uncertainty set $\mathcal{U}$  if and only if there exist $\lambda^0,\ldots,\lambda^l \ge 0$ such that
\[
\left(\begin{array}{ccc}
I_k & I_k \otimes \tilde{x} & A^{(0)}x-a^{(0)} \\
(I_k \otimes \tilde{x})^T & \lambda^0 I_{k(n+1)} & -\lambda^0 b+ \frac{1}{2}\sum_{j=1}^l \lambda^j w^j \\
(A^{(0)}x-a^{(0)})^T & (-\lambda^0 b+ \frac{1}{2}\sum_{j=1}^l \lambda^j w^j)^T & \lambda-\lambda^0(\gamma-\|b\|^2)-\sum_{j=1}^l\lambda^j \beta^j
      \end{array}
 \right)\succeq 0,
\]
where $\tilde{{x}}=({x}^T,-1)^T \in \mathbb{R}^{n+1}$,
${b}={\rm vec}(\overline{\Delta})$ and ${\gamma}={\rho}^2-{\rm Tr}(\overline{\Delta}^T\overline{\Delta})$.
          \end{lemma}
          \begin{proof} Let $\Delta=(\Delta A, \Delta a) \in \mathbb{R}^{k \times n} \times \mathbb{R}^k= \mathbb{R}^{k \times (n+1)}$. For $x \in \mathbb{R}^n$, denote $\tilde{x}=(x^T,-1)^T \in \mathbb{R}^{n+1}$. From the definition of $\mathcal{U}$, we note that
          $\max_{(A,a) \in \mathcal{U}} \|Ax-a\|^2 \le \lambda$ if and only if
\[
\|\Delta-\overline{\Delta}\|_F^2 \le \rho^2, \  (w^{j})^T{\rm vec}\Delta \le \beta^j, j=1,\ldots,l \ \Rightarrow \ \|A^{(0)} x-a^{(0)}+\Delta  \tilde{x}\|^2 \le \lambda,
\]
which is equivalent to the following implication
\begin{eqnarray*}
 & & {\rm Tr}\big(\Delta^T\Delta-2\overline{\Delta}^T\Delta+\overline{\Delta}^T\overline{\Delta}\big) \le {\rho}^2, \   (w^{j})^T{\rm vec}\Delta \le \beta^j,\, j=1,\ldots,l \\
& \Rightarrow & {\rm Tr}\big(\Delta\tilde{x}\, \tilde{x}^T\Delta^T+2(A^{(0)}x-a^{(0)})\tilde{x}^T\Delta+(A^{(0)}x-a^{(0)})(A^{(0)}x-a^{(0)})^T\big)-\lambda \le 0.
\end{eqnarray*}
% Recall the following identities: for matrix $A,C \in \mathbb{R}^{p \times s}$ and $B \in \mathbb{R}^{p \times p}$,
% \[
% {\rm Tr}(A^TBA)={\rm vec}(A)^T(I_s \otimes B) {\rm vec}(A) \mbox{ and } {\rm Tr}(A^TC)={\rm vec}(A)^T {\rm vec}(C)
% \]
% where ${\rm vec}(M)$ denotes the vector obtained by stacking the columns of $M$ and $$I_s \otimes B=\left(\begin{array}{ccccc}
% % c_0 & c_{n-1} & \ldots & c_{2} & c_{1} \\
% % c_{1} & c_0 & c_{n-1} & \ddots & c_2 \\
% % \vdots & \ddots & \ddots & \ddots & \vdots \\
% % c_{n-2} &  & \ddots & c_0 & c_{n-1}\\
% % c_{n-1} & c_{n-2} & \ldots & c_1 & c_0
% B & 0 & 0 & 0 & 0 \\
% 0 & B & 0 & \ddots & 0 \\
% \vdots & \ddots & \ddots & \ddots & \vdots \\
% 0 & & \ddots & B & 0 \\
% 0 & 0 & \ldots & 0 &B
%           \end{array}
%   \right) \in \mathbb{R}^{sp \times sp}. $$
Note that, for matrix $A,C \in \mathbb{R}^{p \times s}$ and $B \in \mathbb{R}^{p \times p}$,
\begin{equation}\label{eq:indentity0}
{\rm Tr}(A^TBA)={\rm vec}(A)^T(I_s \otimes B) {\rm vec}(A) \mbox{ and } {\rm Tr}(A^TC)={\rm vec}(A)^T {\rm vec}(C).
\end{equation}
 Let $u={\rm vec}(\Delta) \in \mathbb{R}^{k(n+1)}$. Then, using the identities in (\ref{eq:indentity0}), we see that $\max_{(A,a) \in \mathcal{U}} \|Ax-a\|^2  \le \lambda$ if and only if the following implication holds
\[
\|u-b\|^2 \le \gamma, (w^j)^Tu \le \beta^j, \, j=1,\ldots,l \ \Rightarrow \ u^TQu+a^Tu +(r+\lambda) \ge 0
\]
where $Q=-(I_{k} \otimes \tilde{x}\tilde{x}^T)$, $q=-{\rm vec}(2\tilde{x}(A^{(0)}x-a^{(0)})^T)$, $r=-{\rm Tr}((A^{(0)}x-a^{(0)})(A^{(0)}x-a^{(0)})^T)$,
${b}={\rm vec}(\overline{\Delta})$ and ${\gamma}={\rho}^2-{\rm Tr}(\overline{\Delta}^T\overline{\Delta})$.  As  $Q=-(I_{k} \otimes \tilde{x}\tilde{x}^T)$, and so,
${\rm dim}{\rm Ker}(Q-\lambda_{\rm min}(Q)I_{k(n+1)}) \ge k \ge s+1.$
\[
{\rm dim}{\rm Ker}(Q-\lambda_{\rm min}(Q)I_{k(n+1)})+{\rm dim}\big(\bigcap_{j=1}^l(w^j)^{\bot}\big) \ge (s+1)+(k(n+1)-s) \ge k(n+1)+1,
\]
where $k(n+1)$ is the dimension of the given matrix data. Since $\mathcal{U}$ has a nonempty interior, by the extended version of S-lemma (Corollary \ref{cor:S_lemma}), we see that $\max_{(A,a) \in \mathcal{U}} \|Ax-a\|^2 \le \lambda$  if and only if there exist $\lambda^0,\lambda^1,\ldots,\lambda^l \ge 0$ such that for all
$u \in \mathbb{R}^{k(n+1)}$,
\[
(u^TQu+q^Tu +r+ \lambda)+\lambda^{0}(\|u-{b}\|^2-{\gamma})+\sum_{j=1}^l \lambda^j((w^{j})^T u-\beta^j) \ge 0
\]
which is equivalent to
\begin{equation}\label{eq:LMI}
\left(\begin{array}{cc}
Q+\lambda^0 I_{k(n+1)} & \frac{1}{2}\big(q-2\lambda^0 {b}+ \sum_{j=1}^l \lambda^j w^j\big) \\
\frac{1}{2}(q-2\lambda^0 {b}+ \sum_{j=1}^l \lambda^j w^j)^T & r+\lambda-\lambda^0({\gamma}-\|{b}\|^2)-\sum_{j=1}^l\lambda^j \beta^j
      \end{array}
 \right)\succeq 0.
\end{equation}
We now apply the method of Schur complement that, for $M_i\in S^n, \ i=1,2,3$ with $M_1\succ 0$,  $\left(\begin{array}{cc} M_1 & M_2 \\
 M_2^T & M_3
 \end{array}\right) \succeq 0 \Leftrightarrow M_3-M_2^TM_1^{-1}M_2 \succeq 0$, to reformulate (\ref{eq:LMI}) into linear matrix inequalities. To see this, note that
\begin{eqnarray*}{}
 &&Q=-(I_{k} \otimes \tilde{x}\tilde{x}^T)=-(I_{k} \otimes \tilde{x})(I_k \otimes \tilde{x})^T \\
 && q=-{\rm vec}(2\tilde{x}(A^{(0)}x-a^{(0)})^T)=-2(I_k \otimes \tilde{x})(A^{(0)}x-a^{(0)}) \\
 && r=-{\rm Tr}((A^{(0)}x-a^{(0)})(A^{(0)}x-a^{(0)})^T)=-\|A^{(0)}x-a^{(0)}\|^2,
\end{eqnarray*}
and let $M_1=I_k$, $M_2=(I_k \otimes \tilde{x}, A^{(0)}x-a^{(0)})$ and $$M_3=\left(\begin{array}{cc}
 \lambda^0 I_{k(n+1)} & -\lambda^0  {b}+ \frac{1}{2}\sum_{j=1}^l \lambda^j w^j \\
 (-\lambda^0  b+ \frac{1}{2}\sum_{j=1}^l \lambda^j w^j)^T & \lambda-\lambda^0({\gamma}-\|{b}\|^2)-\sum_{j=1}^l\lambda^j \beta^j \end{array}\right).$$
Then, $\max_{(A,a) \in \mathcal{U}} \|Ax-a\|^2 \le \lambda$ is equivalent to the following linear matrix inequality problem:   there exist $\lambda^0,\ldots,\lambda^l \ge 0$ such that
\[
\left(\begin{array}{ccc}
I_k & I_k \otimes \tilde{x} & A^{(0)}x-a^{(0)} \\
(I_k \otimes \tilde{x})^T & \lambda^0 I_{k(n+1)} & -\lambda^0 b+ \frac{1}{2}\sum_{j=1}^l \lambda^j w^j \\
(A^{(0)}x-a^{(0)})^T & (-\lambda^0 b+ \frac{1}{2}\sum_{j=1}^l \lambda^j w^j)^T & \lambda-\lambda^0(\gamma-\|b\|^2)-\sum_{j=1}^l\lambda^j \beta^j
      \end{array}
 \right)\succeq 0.
\]
          \end{proof}
\begin{remark}{\bf (Key to SDP reformulation)}{\rm
The key to the SDP reformulation in Lemma 5.1 is that the robust feasibility
of a given point can be equivalently rewritten as a quadratic optimization problem where the Hessian of the objective function is $-I_{k} \otimes \tilde{x}\tilde{x}^T$ (which
has at least multiplicity $k$ for each of its eigenvalues). So, the assumption that $k\ge s+1$ guarantees our dimension condition. This enables us to convert the robust problem into a SDP using our S-lemma. This technique has been exploited and used in robust optimization recently, see \cite{beck-eldar,Beck0}.}
\end{remark}

\subsection{Robust Least Squares}
Consider the least squares problem (LSP) under data uncertainty (see \cite{robust_LS})
\begin{eqnarray*}
(LSP)& \displaystyle \min_{x \in \mathbb{R}^n} & \|Ax-a\|^2
\end{eqnarray*}
where the data $(A,a) \in \mathbb{R}^{k \times n} \times \mathbb{R}^k$ is uncertain and it belongs to the matrix  uncertainty set $\mathcal{U}$.  The robust counterpart of the uncertain least squares problem can be stated as follows:
\begin{eqnarray*}
(RLSP) &  \displaystyle  \min_{x \in \mathbb{R}^n}\max_{(A,a) \in \mathcal{U}} & \|Ax-a\|^2,
\end{eqnarray*}
which seeks a solution $x \in \mathbb{R}^n$ that minimizes the worst case data error with
respect to all possible values of $(A,a) \in \mathcal{U}$.

The tractability of the robust problem (RSLP) strongly relies on the choice of the uncertainty set
$\mathcal{U}$. For example, if the uncertainty set $\mathcal{U}$ is described by a single ellipsoid then (RSLP)
 can be reformulated as a semidefinite programming problem, and so, is tractable (see El Ghaoui and Lebretis \cite{robust_LS}).
 Also, if $\mathcal{U}$ is given by an intersection of two ellipsoids, (RSLP) can be  reformulated as a semidefinite programming problem under
 suitable regularity conditions (see \cite{beck-eldar}). However, if the uncertainty set $\mathcal{U}$ is given by an intersection of
 finitely many, but more than two, ellipsoids, then (RSLP) is generally not tractable (see \cite{BNR}). %On the other hand,

Here, we provide a new tractable case where the uncertainty is $\mathcal{U}$ is given
by (\ref{eq:pp0}). %
% described as intersection of matrix norm constraints and polyhedral constraints, i.e., $\$
%\begin{equation}\label{eq:pp0}
%\mathcal{U}=\{\tilde{A}^{(0)} + \Delta: \Delta \in \mathbb{R}^{k \times (n+1)}, \|\Delta-\overline{\Delta}\|_F \le \rho, \  (w^{j})^T{\rm vec}\Delta \le \beta^j, j=1,\ldots,l \},
%\end{equation}
%with $\tilde{A}^{(0)}:=(A^{(0)},a^{(0)}) \in \mathbb{R}^{k \times n} \times \mathbb{R}^k= \mathbb{R}^{k \times (n+1)}$
% and $\|M\|_F$ is the Frobenius norm defined by $\|M\|_F=\sqrt{{\rm Tr}(M^TM)}$.

%{\it Here, we also assume that $k \ge 2$ and the uncertainty set has nonempty interior, i.e., $$\{\Delta_i: \|\Delta_i-\overline{\Delta}_i\|_F < \overline{\rho}_i, \  \|\Delta_i-\underline{\Delta}_i\|_F < \underline{\rho}_i\} \neq \emptyset.$$}
\begin{theorem}{\bf (SDP characterization of (RSLP) solution)} \label{th:robust_fractional}
Let $x \in \mathbb{R}^n$. For problem (RSLP) with $\mathcal{U}$ defined as in (\ref{eq:pp0}), assume that  $k \ge s+1$, where $k$ is the number of rows in the matrix data of $\mathcal{U}$ and $s={\rm dim}\, {\rm span}\{w^1,\ldots,w^l\}$, and that $\{\Delta: \|\Delta-\overline{\Delta}\|_F < \rho, \  (w^{j})^T{\rm vec}\Delta < \beta^j, j=1,\ldots,l\} \neq \emptyset.$.  Then $x$ solves (RLSP) if and only if  $(x,\lambda,\lambda^0,\ldots,\lambda^l) \in \mathbb{R}^n \times \mathbb{R} \times \mathbb{R}_+ \times \ldots \times \mathbb{R}_+$  solves the following linear semi-definite programming problem:
\begin{eqnarray*}
& &  \min_{(x,\lambda) \in \mathbb{R}^n\times \mathbb{R}, \lambda^0,\ldots,\lambda^l \ge 0}\{\lambda: \\
& & \left(\begin{array}{ccc}
I_k & I_k \otimes \tilde{x} & A^{(0)}x-a^{(0)} \\
(I_k \otimes \tilde{x})^T & \lambda^0 I_{k(n+1)} & -\lambda^0 b+ \frac{1}{2}\sum_{j=1}^l \lambda^j w^j \\
(A^{(0)}x-a^{(0)})^T & (-\lambda^0 b+ \frac{1}{2}\sum_{j=1}^l \lambda^j w^j)^T & \lambda-\lambda^0(\gamma-\|b\|^2)-\sum_{j=1}^l\lambda^j \beta^j
      \end{array}
 \right)\succeq 0\},
\end{eqnarray*}
for some $\lambda \in \mathbb{R}$ and $\lambda^0,\ldots,\lambda^l \ge 0$, where $\tilde{{x}}=({x}^T,-1)^T \in \mathbb{R}^{n+1}$,
${b}={\rm vec}(\overline{\Delta})$ and ${\gamma}={\rho}^2-{\rm Tr}(\overline{\Delta}^T\overline{\Delta})$.
\end{theorem}
\begin{proof}
%
%
%
% which is equivalent to
% \begin{equation} \label{eq:MI}
% \left(\begin{array}{cc}
% 2(Q+\lambda_1 I_{k(n+1)}+\lambda_2 I_{k(n+1)}) & q+\lambda_1 \overline{b}+\lambda_2 \underline{b} \\
% (q+\lambda_1 \overline{b}+\lambda_2 \underline{b})^T & 2(r+\lambda s+\lambda_1\overline{\gamma}+\lambda_2\underline{\gamma})
%       \end{array}
%  \right)\succeq 0.
% \end{equation}
% Since
% \begin{eqnarray*}{}
%  &&Q=-(I_{k} \otimes \tilde{x}\tilde{x}^T)=-(I_{k} \otimes \tilde{x})(I_k \otimes \tilde{x})^T \\
%  && q=-{\rm vec}(2\tilde{x}(A^{(0)}x-a^{(0)})^T)=-2(I_k \otimes \tilde{x})(A^{(0)}x-a^{(0)}) \\
%  && r=-{\rm Tr}((A^{(0)}x-a^{(0)})(A^{(0)}x-a^{(0)})^T)=-\|A^{(0)}x-a^{(0)}\|^2,
% \end{eqnarray*}
% using the same Schur complement argument as in the proof of Theorem \ref{th:SOCP}, we can rewrite (\ref{eq:MI}) as the following problem: there exist $\lambda_1,\lambda_2 \ge 0$ such that
% \[
% \left(\begin{array}{ccc}
% I_k & I_k \otimes \tilde{x} & A^{(0)}x-a^{(0)} \\
% (I_k \otimes \tilde{x})^T & \lambda_1 I_{k(n+1)}+\lambda_2 I_{k(n+1)} & 0 \\
% (A^{(0)}x-a^{(0)})^T & 0 & 2(\lambda s+\lambda_1\overline{\gamma}+\lambda_2\underline{\gamma})
%       \end{array}
%  \right)\succeq 0.
% \]
Note that $x$ is a solution of $\displaystyle \min_{x \in \mathbb{R}^n}\max_{(A,a) \in \mathcal{U}}  \|Ax-a\|^2$ if and only if there exists $\lambda \in \mathbb{R}$ such that $(x,\lambda)$
solves $\displaystyle \min_{(x,\lambda) \in \mathbb{R}^n\times \mathbb{R}} \{\lambda: \displaystyle \max_{(A,a) \in \mathcal{U}} \|Ax-a\|^2 \le \lambda\}$.
 %$x$ is a solution of (RLSP) if and only if there exist there exist $\lambda^0,\ldots,\lambda^l \ge 0$ such that such that $(x,\lambda^0,\lambda^1,\ldots,\lambda^l)$ solves the following linear semi-definite programming problem:
Then, by Lemma \ref{lemma:5.1},we see that $x \in \mathbb{R}^n$ solves (RLSP) if and only if  $(x,\lambda,\lambda^0,\ldots,\lambda^l) \in \mathbb{R}^n \times \mathbb{R} \times \mathbb{R}_+ \times \ldots \times \mathbb{R}_+$  solves the following linear semi-definite programming problem:
\begin{eqnarray*}
& &  \min_{(x,\lambda) \in \mathbb{R}^n\times \mathbb{R}, \lambda^0,\ldots,\lambda^l \ge 0}\{\lambda: \\
& & \left(\begin{array}{ccc}
I_k & I_k \otimes \tilde{x} & A^{(0)}x-a^{(0)} \\
(I_k \otimes \tilde{x})^T & \lambda^0 I_{k(n+1)} & -\lambda^0 b+ \frac{1}{2}\sum_{j=1}^l \lambda^j w^j \\
(A^{(0)}x-a^{(0)})^T & (-\lambda^0 b+ \frac{1}{2}\sum_{j=1}^l \lambda^j w^j)^T & \lambda-\lambda^0(\gamma-\|b\|^2)-\sum_{j=1}^l\lambda^j \beta^j
      \end{array}
 \right)\succeq 0\}
\end{eqnarray*}
for some $\lambda \in \mathbb{R}$ and $\lambda^0,\ldots,\lambda^l \ge 0$.
\end{proof}

Consider the special case of the uncertainty set, $\mathcal{U}$, in (\ref{eq:pp0}) where $l=1$, $\overline{\Delta}=0$, $w^1=0$ and $\beta^1 =1$. In this case, the $\mathcal{U}$ reduces to the matrix norm uncertainty set of the form
\begin{eqnarray}\label{eq:pp6}
\mathcal{U} & = & \{\tilde{A}^{(0)} + \Delta: \Delta \in \mathbb{R}^{k \times (n+1)}, \|\Delta\|_F \le \rho\},
\end{eqnarray}
and the tractability of robust least squares problem (RLSP) was established in El~Ghoui et al. \cite{robust_LS}. In the following Corollary we derive an SDP characterization of (RLSP) for the uncertainty set (\ref{eq:pp6}).
\begin{corollary}{\bf (Matrix norm uncertainty)}
Let $x \in \mathbb{R}^n$. For problem (RSLP) with $\mathcal{U}$ defined as in (\ref{eq:pp6}), assume that $\rho>0$.  Then $x$ solves (RLSP) if and only if  $(x,\lambda,\lambda^0,\lambda^1) \in \mathbb{R}^n \times \mathbb{R} \times \mathbb{R}_+ \times  \mathbb{R}_+$  solves the following linear semi-definite programming problem:
\begin{eqnarray*}
& &  \min_{(x,\lambda) \in \mathbb{R}^n\times \mathbb{R}, \lambda^0,\lambda^1 \ge 0}\{\lambda: \\
& & \left(\begin{array}{ccc}
I_k & I_k \otimes \tilde{x} & A^{(0)}x-a^{(0)} \\
(I_k \otimes \tilde{x})^T & \lambda^0 I_{k(n+1)} &  0 \\
(A^{(0)}x-a^{(0)})^T & 0 & \lambda-\lambda^0\rho^2-\lambda^1
      \end{array}
 \right)\succeq 0\}.
\end{eqnarray*}
for some $\lambda \in \mathbb{R}$ and $\lambda^0,\lambda^1 \ge 0$, where $\tilde{{x}}=({x}^T,-1)^T \in \mathbb{R}^{n+1}$.
\end{corollary}
\begin{proof}
Let $l=1$, $\overline{\Delta}=0$, $w^1=0$ and $\beta^1 =1$.  Then, $s={\rm dim}{\rm span}\{w^1\} =0$, and so, $k \ge 1 = s+1$. Moreover, as $\rho>0$, the strict feasibility condition is satisfied for $\Delta=0$. Thus, the conclusion follows by the preceding theorem.
\end{proof}
Consider the special case of the uncertainty set, $\mathcal{U}$, in (\ref{eq:pp0}), where $l=2$, $\overline{\Delta}=0$, $w^2=-w^1$ and $\beta^1=-\beta^2 =1$. In this case, $\mathcal{U}$ simplifies to case of an intersection of two ellipsoids of the form
\begin{eqnarray}\label{eq:pp9}
\mathcal{U} & = & \{\tilde{A}^{(0)} + \Delta: \Delta \in \mathbb{R}^{k \times (n+1)}, \|\Delta\|_F \le \rho, \  -1 \le (w^{1})^T{\rm vec}\Delta \le 1\} \nonumber \\
& = & \{\tilde{A}^{(0)} + \Delta: \Delta \in \mathbb{R}^{k \times (n+1)}, {\rm Tr}(\Delta^T\Delta) \le \rho^2, \  {\rm Tr}(\Delta^T B \Delta) \le 1\},
\end{eqnarray}
where $B=(w^1)(w^1)^T$. In this case, an SDP characterization of robust solution was established in Beck and Eldar \cite{beck-eldar}. In this case we obtain the following corollary.
\begin{corollary}{\bf (Intersection of two ellipsoids uncertainty)}
Let $x \in \mathbb{R}^n$. For problem (RSLP) with $\mathcal{U}$ defined as in (\ref{eq:pp9}), assume that  $k \ge 2$, where $k$ is the number of rows in the matrix data of $\mathcal{U}$, and that $\rho>0$.  Then $x$ solves (RLSP) if and only if  $(x,\lambda,\lambda^0,\lambda^1,\lambda^2) \in \mathbb{R}^n \times \mathbb{R} \times \mathbb{R}_+ \times  \mathbb{R}_+$  solves the following linear semi-definite programming problem:
\begin{eqnarray*}
& &  \min_{(x,\lambda) \in \mathbb{R}^n\times \mathbb{R}, \lambda^0,\ldots,\lambda^l \ge 0}\{\lambda: \\
& & \left(\begin{array}{ccc}
I_k & I_k \otimes \tilde{x} & A^{(0)}x-a^{(0)} \\
(I_k \otimes \tilde{x})^T & \lambda^0 I_{k(n+1)} &  \frac{1}{2}(\lambda^1 w^1-\lambda^2 w^1) \\
(A^{(0)}x-a^{(0)})^T & \frac{1}{2}(\lambda^1 w^1-\lambda^2 w^1)^T & \lambda-\lambda^0\rho^2-(\lambda^1 -\lambda^2)
      \end{array}
 \right)\succeq 0\},
\end{eqnarray*}
for some $\lambda \in \mathbb{R}$ and $\lambda^0,\lambda^l,\lambda^2 \ge 0$, where $\tilde{{x}}=({x}^T,-1)^T \in \mathbb{R}^{n+1}$.
\end{corollary}
\begin{proof}
Let $l=2$, $\overline{\Delta}=0$, $w^2=-w^1$ and $\beta^1=-\beta^2 =1$. Then, $s={\rm dim}{\rm span}\{w^1,w^2\} \le 1$, and so, $k \ge 2 \ge s+1$. Moreover, as $\rho>0$, the strict feasibility condition is satisfied for $\Delta=0$. Thus, the conclusion follows from Theorem 5.1.
\end{proof}
\begin{remark} {\rm {\bf (Tractability of (RLSP))}
It follows easily from Theorem 5.1 that finding a  solution of the
robust least squares with the uncertainty set
given by an intersection of a norm constraint and a polyhedral constraint is equivalent
to solving a linear semi-definite programming problem.
Note that a linear semi-definite programming problem can be solved in polynomial time and $s={\rm dim}{\rm span}\{w^1,\ldots,w^l\} \le l$ (and so, $k \ge l+1$ implies that $k \ge s+1$). So, a solution of this robust least squares can be validated in polynomial time whenever
$k\ge l+1$ where $k$ is the number of rows in the matrix data $A$ and $l$ is the number of the linear inequalities that defines the uncertainty set. }
\end{remark}
%\begin{eqnarray*}
%& &  \min_{(x,\lambda) \in \mathbb{R}^n\times \mathbb{R}, \lambda^0,\ldots,\lambda^l \ge 0}\{\lambda: \\
%& & \left(\begin{array}{ccc}
%I_k & I_k \otimes \tilde{x} & A^{(0)}x-a^{(0)} \\
%(I_k \otimes \tilde{x})^T & \lambda^0 I_{k(n+1)} & -\lambda^0 b+ \frac{1}{2}\sum_{j=1}^l \lambda^j w^j \\
%(A^{(0)}x-a^{(0)})^T & (-\lambda^0 b+ \frac{1}{2}\sum_{j=1}^l \lambda^j w^j)^T & \lambda-\lambda^0(\gamma-\|b\|^2)-\sum_{j=1}^l\lambda^j \beta^j
%      \end{array}
% \right)\succeq 0.
%\end{eqnarray*}

\subsection{Robust Second Order Cone Programming Problems}
%In this section, we establish a dual characterization of the solution of a robust second-order cone programming problem where the uncertainty set is  given by the intersection of the norm constraint and the polyhedral constraint. Consequently, we show that solving the robust second-order cone programming problem is equivalent to solving a semi-definite linear programming problem and so the solution can be validated in polynomial time.

%Let us note first that, for a $(p \times q)$ matrix, ${\rm vec}(M)$ denotes the vector in $\mathbb{R}^{pq}$ obtained
%by stacking the columns of $M$. The tensor product of $I_n$ and a matrix
%$M \in \mathbb{R}^{p \times p}$ is defined by
%$$I_n \otimes M:=\left(\begin{array}{ccccc}
%% c_0 & c_{n-1} & \ldots & c_{2} & c_{1} \\
%% c_{1} & c_0 & c_{n-1} & \ddots & c_2 \\
%% \vdots & \ddots & \ddots & \ddots & \vdots \\
%% c_{n-2} &  & \ddots & c_0 & c_{n-1}\\
%% c_{n-1} & c_{n-2} & \ldots & c_1 & c_0
%M & 0 & 0 & 0 & 0 \\
%0 & M & 0 & \ddots & 0 \\
%\vdots & \ddots & \ddots & \ddots & \vdots \\
%0 & & \ddots & M & 0 \\
%0 & 0 & \ldots & 0 &M
%          \end{array}\right) \in \mathbb{R}^{np \times np} .$$

Consider the linear second-order cone programs (SOCP) (cf. \cite{Second_order}) under constraint data uncertainty
\begin{eqnarray*}
(SOCP)& \min_{x \in \mathbb{R}^n} & a^Tx  \\
& \mbox{ s.t. } & \|B_ix-b_i\| \le d_i, \ i=1,\ldots,m,
\end{eqnarray*}
{where the data $\tilde{B}_i=(B_i,b_i) \in \mathbb{R}^{k_i \times n} \times \mathbb{R}^{k_i}= \mathbb{R}^{k_i \times (n+1)}$, $i=1,\ldots,m$, is uncertain and it belongs to the matrix  uncertainty set $\mathcal{U}_i$.
 The robust counterpart of the uncertain second-order cone problem can be stated as follows:
\begin{eqnarray*}
(RSOCP) & \min_{x \in \mathbb{R}^n} & a^Tx  \\
& \mbox{ s.t. } & \|B_ix-b_i\| \le d_i, \ \forall (B_i,b_i) \in \mathcal{U}_i,\   i=1,\ldots,m.
\end{eqnarray*}
Note that, although the (RSOCP) is, in general, not tractable \cite{BNR} when $\mathcal{U}_i$ is given by an intersection of finitely ellipsoids, recently, Beck \cite{Beck0} has identified an
interesting tractable subclass where
 $\mathcal{U}_i$ is described by  at most $k$ many homogeneous quadratic inequalities
 under a suitable regularity condition. %that is, for each $i=1,\ldots,m$,
% \[
% \mathcal{U}_i= \{(B_i^{(0)},b_i^{(0)}) + \Delta: \Delta \in \mathbb{R}^{k \times (n+1)}, \, \|C_j \Delta^T\|_F^2 \le \rho_j^2,\  j = 1,\ldots,k\},
% \]
% where $C_j \in \mathbb{R}^{(n+1) \times (n+1)}$ such that there exist $\mu_j \in \mathbb{R}$ such that $\sum_{j=1}^k\mu_jC_j^TC_j \succ 0.$}

Here, we examine (RSOCP) in the case where the uncertainty set is given by an  intersection of a matrix norm constraint and polyhedral constraints, i.e.,
\begin{equation}\label{eq:pp}
\mathcal{U}_i=\{\tilde{B}_i^{(0)} + \Delta_i: \Delta_i \in \mathbb{R}^{k_i \times (n+1)}, \|\Delta_i-\overline{\Delta}_i\|_F \le \rho_i, \  (w_i^{j})^T{\rm vec}\Delta_i \le \beta_i^j, j=1,\ldots,l_i \},
\end{equation}
with $\tilde{B}_i^{(0)}:=(B_i^{(0)},b_i^{(0)}) \in \mathbb{R}^{k_i \times n} \times \mathbb{R}^k= \mathbb{R}^{k_i \times (n+1)}$
 and $\|M\|_F$ is the Frobenius norm defined by $\|M\|_F=\sqrt{{\rm Tr}(M^TM)}$. We denote $s_i={\rm dim}\, {\rm span}\{w_i^1,\ldots,w_i^l\}$, $i=1,\ldots,m$.

We note that our model differs from the model considered in \cite{Beck0} because a polyhedral set in $\mathbb{R}^{(n+1) \times (n+1)}$ cannot, in general, be described as the finite intersection of sets of the form $\{\Delta \in \mathbb{R}^{(n+1) \times (n+1)}:\|C_j \Delta^T\|_F^2 \le \rho_j^2\}$ in general.

%We say that $\overline{x}$ is a robust feasible point (resp. robust solution) of the uncertain second-order cone problem  if it is a feasible point (resp. solution)
%of the robust counterpart (RSOCP).

\begin{theorem} \label{th:SOCP} {\bf (SDP characterization of (RSOCP) solution)} For problem (RSOCP) with $\mathcal{U}_i$ defined as in (\ref{eq:pp}). Assume that, for each $i=1,\ldots,m$, $k_i \ge s_i+1$,  and that $\{\Delta_i: \|\Delta_i-\overline{\Delta}_i\|_F < \rho_i, \  (w_i^{j})^T{\rm vec}\Delta_i < \beta_i^j, j=1,\ldots,l_i\} \neq \emptyset.$
 A point $x\in \mathbb{R}^n$ solves (RSOCP) if and only if $(x,\lambda_1,\ldots,\lambda_m) \in \mathbb{R}^n  \times \mathbb{R}^{l_1+1}_+ \times \ldots \times \mathbb{R}^{l_m+1}_+$  solves the following linear semi-definite programming problem:
\begin{eqnarray*}
& & \min_{\lambda_i^0,\ldots,\lambda_i^{l_i} \ge 0, x \in \mathbb{R}^n}\{a^Tx: \\
& & \left(\begin{array}{ccc}
I_{k_i} & I_{k_i} \otimes \tilde{x} & B_i^{(0)}x-b_i^{(0)} \\
(I_{k_i} \otimes \tilde{x})^T & \lambda_i^0 I_{k_i(n+1)} & -\lambda_i^0 \overline{b}_i+ \frac{1}{2}\sum_{j=1}^{l_i} \lambda_i^j w_i^j \\
(B_i^{(0)}x-b_i^{(0)})^T & (-\lambda_i^0 \overline{b}_i+ \frac{1}{2}\sum_{j=1}^{l_i} \lambda_i^j w_i^j)^T & d_i^2-\lambda_i^0(\overline{\gamma}_i-\|\overline{b}_i\|^2)-\sum_{j=1}^{l_i}\lambda_i^j \beta_i^j
      \end{array}
 \right)\succeq 0\},
\end{eqnarray*}
for some $\lambda_i=(\lambda_i^0,\lambda_i^1,\ldots,\lambda_i^{l_i}) \in \mathbb{R}^{l_i+1}_+$, $i=1,\ldots,m$, where $\tilde{{x}}=({x}^T,-1)^T \in \mathbb{R}^{n+1}$ and
$\overline{b}_i={\rm vec}(\overline{\Delta}_i)$ and $\overline{\gamma}_i=\rho_i^2$.
\end{theorem}
\begin{proof}
%Let $\Delta_i=(\Delta B_i, \Delta b_i) \in \mathbb{R}^{k_i \times n} \times \mathbb{R}^{k_i}= \mathbb{R}^{k_i \times (n+1)}$, $i=1,\ldots,m$.
Note that a point $x$ is robust feasible if, for all $i=1,\ldots,m$,
\[
\max_{(B_i,b_i)\in \mathcal{U}_i}\|B_ix-b_i\|^2 \le d_i^2.
\]
So, Lemma \ref{lemma:5.1} implies that
the robust feasibility of $x$ can be  equivalently rewritten as the following linear matrix inequality problem: for all $i=1,\ldots,m$,  there exist $\lambda_i^0,\ldots,\lambda_i^{l_i} \ge 0$ such that
\[
\left(\begin{array}{ccc}
I_{k_i} & I_{k_i} \otimes \tilde{x} & B_i^{(0)}x-b_i^{(0)} \\
(I_{k_i} \otimes \tilde{x})^T & \lambda_i^0 I_{k_i(n+1)} & -\lambda_i^0 \overline{b}_i+ \frac{1}{2}\sum_{j=1}^{l_i} \lambda_i^j w_i^j \\
(B_i^{(0)}x-b_i^{(0)})^T & (-\lambda_i^0 \overline{b}_i+ \frac{1}{2}\sum_{j=1}^{l_i} \lambda_i^j w_i^j)^T & d_i^2-\lambda_i^0(\overline{\gamma}_i-\|\overline{b}_i\|^2)-\sum_{j=1}^{l_i}\lambda_i^j \beta_i^j
      \end{array}
 \right)\succeq 0.
\]
Thus, the conclusion follows.
 \end{proof}
Consider the special case of the uncertainty set (\ref{eq:pp}), where $l_i=1$, $\overline{\Delta}=0$, $w_i^1=0$ and $\beta_i^1 =1$, $i=1,\ldots,m$. In this case $\mathcal{U}_i$ reduces to the matrix norm uncertainty set of the form
\begin{eqnarray}\label{eq:pp7}
\mathcal{U}_i & = & \{\tilde{A}_i^{(0)} + \Delta_i: \Delta_i \in \mathbb{R}^{k_i \times (n+1)}, \|\Delta_i\|_F \le \rho\}.
\end{eqnarray}
An SDP characterization of robust solution of second-order cone programming problem was established in  \cite{BNR}.
\begin{corollary}{\bf (Matrix norm uncertainty)}
Let $x \in \mathbb{R}^n$. For problem (RSOCP) with $\mathcal{U}$ defined as in (\ref{eq:pp7}), assume that $\rho>0$.  Then $x$ solves (RSOCP) if and only if  $(x,\lambda,\lambda_i^0,\lambda_i^1) \in \mathbb{R}^n \times \mathbb{R} \times \mathbb{R}_+ \times  \mathbb{R}_+$  solves the following linear semi-definite programming problem:
\begin{eqnarray*}
& &  \min_{(x,\lambda) \in \mathbb{R}^n\times \mathbb{R}, \lambda_i^0,\lambda_i^1 \ge 0}\{\lambda: \\
& & \left(\begin{array}{ccc}
I_k & I_k \otimes \tilde{x} & A^{(0)}x-a^{(0)} \\
(I_k \otimes \tilde{x})^T & \lambda^0 I_{k_i(n+1)} &  0 \\
(A^{(0)}x-a^{(0)})^T & 0 & \lambda-\lambda_i^0\rho^2-\lambda_i^1
      \end{array}
 \right)\succeq 0\},
\end{eqnarray*}
for some $\lambda \in \mathbb{R}$ and $\lambda_i^0,\lambda_i^l \ge 0$, where $\tilde{{x}}=({x}^T,-1)^T \in \mathbb{R}^{n+1}$.
\end{corollary}
\begin{proof}
Let $l_i=1$, $\overline{\Delta}=0$, $w_i^1=0$ and $\beta_i^1 =1$, $i=1,\ldots,m$. Then, $s_i={\rm dim}{\rm span}\{w_i^1\} =0$, and so, $k_i \ge 1 \ge s_i+1$. Moreover, as $\rho>0$, the strict feasibility condition is satisfied for $\Delta=0$. Thus, the conclusion follows from Theorem 5.2.
\end{proof}

Consider another special case of the uncertainty set (\ref{eq:pp}), where $k_i=k-1$, $\overline{\Delta}_i=0$, $\beta_i^l=-\beta_i^{l+k-1} =1$ and $w_i^l=-w_i^{l+k-1}=w^l$, $l=1,\ldots,2k_i$, $i=1,\ldots,m$. In this case, the uncertainty set $\mathcal{U}_i$ simplifies to the intersection of $k$ many ellipsoids of the form
{\small\begin{eqnarray}\label{eq:pp99}
 \mathcal{U}_i  =&  \{\tilde{A}_i^{(0)} + \Delta_i: \Delta_i \in \mathbb{R}^{k \times (n+1)}, \|\Delta_i\|_F \le \rho_i,  -1 \le (w^{l})^T{\rm vec}\Delta_i \le 1, l=1,\ldots,k-1\} \nonumber \\
 = &\{\tilde{A}_i^{(0)} + \Delta_i: \Delta_i \in \mathbb{R}^{k \times (n+1)}, {\rm Tr}(\Delta_i^T\Delta_i) \le \rho_i^2,  {\rm Tr}(\Delta_i^T C^l \Delta_i) \le 1, l=1,\ldots,k-1\},
\end{eqnarray}}where $C^l=(w^l)(w^l)^T$, $l=1,\ldots,k-1$. The following robust solution characterization in terms of SDP has been given in Beck \cite{Beck0}.

\begin{corollary} {\rm \cite[Section 4.3]{Beck0}}{\bf (Intersection of many ellipsoids uncertainty)}
Let $x \in \mathbb{R}^n$. For problem (RSOCP) with $\mathcal{U}_i$ defined as in (\ref{eq:pp99}), assume that  $\rho_i>0$.  A point $x\in \mathbb{R}^n$ solves (RSOCP) if and only if $(x,\lambda_1,\ldots,\lambda_m) \in \mathbb{R}^n  \times \mathbb{R}^{2k-1}_+ \times \ldots \times \mathbb{R}^{2k-1}_+$  solves the following linear semi-definite programming problem:
{\small \begin{eqnarray*}
& & \min_{\lambda_i^0,\ldots,\lambda_i^{2(k-1)} \ge 0, x \in \mathbb{R}^n}\{a^Tx: \\
& &  \left(\begin{array}{ccc}
I_{k_i} & I_{k_i} \otimes \tilde{x} & B_i^{(0)}x-b_i^{(0)} \\
(I_{k_i} \otimes \tilde{x})^T & \lambda_i^0 I_{k_i(n+1)} &\displaystyle \frac{1}{2}(\sum_{j=1}^{k-1} \lambda_i^j w^j-\sum_{j=1}^{k-1} \lambda_i^{j+k-1} w^j) \\
(B_i^{(0)}x-b_i^{(0)})^T & \displaystyle \frac{1}{2}(\sum_{j=1}^{k-1} \lambda_i^j w^j-\sum_{j=1}^{k-1} \lambda_i^{j+k-1} w^j)^T & \displaystyle d_i^2-\lambda_i^0\rho_i^2-\sum_{j=1}^{k-1} \lambda_i^j+\sum_{j=1}^{k-1} \lambda_i^{j+k-1}
      \end{array}
 \right)\succeq 0\},
\end{eqnarray*}}
for some $\lambda_i=(\lambda_i^0,\lambda_i^1,\ldots,\lambda_i^{2(k-1)}) \in \mathbb{R}^{2k-1}_+$, $i=1,\ldots,m$, where $\tilde{{x}}=({x}^T,-1)^T \in \mathbb{R}^{n+1}$.
\end{corollary}
\begin{proof}
Let $k_i=k-1$, $\overline{\Delta}_i=0$, $w_i^l=-w_i^{l+k-1}=w^l$ and $\beta_i^l=-\beta_i^{l+k-1} =1$, $l=1,\ldots,2k_i$,  $i=1,\ldots,m$. Then, $s_i={\rm dim} \, {\rm span}\{w^1,\ldots,w^{2(k-1)}\} \le k-1$, and so, $k_i=k \ge s_i+1$. Moreover, as $\rho_i>0$, the strict feasibility condition is satisfied for $\Delta=0$. Thus, the conclusion follows by the preceding theorem.
\end{proof}
%\begin{remark}  {\rm {\bf (Tractability of (RSOCP))}
%From the previous theorem, we see that finding a  solution of the  robust second-order cone programming problem with uncertainty set given by intersection of norm constraint and polyhedral constraint is equivalent to solving a linear semi-definite programming problem.
%%\begin{eqnarray*}
%%& & \min_{\lambda_i^0,\ldots,\lambda_i^l \ge 0, x \in \mathbb{R}^n}\{a^Tx: \\
%%& & \left(\begin{array}{ccc}
%%I_{k_i} & I_{k_i} \otimes \tilde{x} & B_i^{(0)}x-b_i^{(0)} \\
%%(I_{k_i} \otimes \tilde{x})^T & \lambda_i^0 I_{k_i(n+1)} & -\lambda_i^0 \overline{b}_i+ \frac{1}{2}\sum_{j=1}^{l_i} \lambda_i^j w_i^j \\
%%(B_i^{(0)}x-b_i^{(0)})^T & (-\lambda_i^0 \overline{b}_i+ \frac{1}{2}\sum_{j=1}^{l_i} \lambda_i^j w_i^j)^T & d_i^2-\lambda_i^0(\overline{\gamma}_i-\|\overline{b}_i\|^2)-\sum_{j=1}^{l_i}\lambda_i^j \beta_i^j
%%      \end{array}
%% \right)\succeq 0.
%%\end{eqnarray*}
%Note that a linear semi-definite programming problem can be solved in polynomial time. So, the solution of this robust second-order cone programming problem with uncertainty set that is given by intersection of norm constraint and polyhedral constraint, can be validated in polynomial time whenever
%$k_i>l_i+1$ where $k_i$ is the number of rows in the matrix data $B_i$ and $l_i$ is the number of the linear inequalities that defines the uncertainty set $\mathcal{U}_i$.
%}\end{remark}

\setcounter{equation}{0}
\section{Extensions and Further Research}
In this section, we present how our approach extends to more general trust-region problems that incorporate uniform convex quadratic inequalities. To examine this, consider the system of quadratic functions, $f(x)=x^TAx+a^Tx+\gamma$, $g_0(x)=\|x-x_0\|^2-\alpha$ and $g_i(x)=\|Bx\|^2+b_i^Tx-\beta_i$, $i=1,\ldots,m$,
 where $A \in S^{n \times n}$, $B \in \mathbb{R}^{l \times n}$ with $l \in \mathbb{N}$,
 $a,x_0,b_i \in \mathbb{R}^n$ and $\gamma,\alpha,\beta_i \in \mathbb{R}$. In this case, we can consider the following {\it extended dimension condition}
\begin{equation}\label{eq:EDC}
 {\rm dim}\big({\rm Ker}(A -\lambda_{\rm min}(A)I_n) \cap {\rm Ker}(B)\big) \ge s +1,
\end{equation}
where $s$ is the dimension of ${\rm span}\{b_1,\ldots,b_m\}$.

Clearly, if the matrix $B$ is zero, then the above quadratic systems and the extended dimension condition reduce to the quadratic systems and its associated dimension condition, studied in Sections 2-4.
 On the other hand, in the case when $B$ has rank $n$,  our dimension condition (\ref{eq:EDC}) fails.

As we see in the following Proposition, the hidden convexity property of Section 2 follows for the above general quadratic system under the extended dimension condition.
\begin{proposition}{\bf (Hidden convexity of General Quadratic Systems)}\label{prop:2.20}
Let $f(x)=x^TAx+a^Tx+\gamma$, $g_0(x)=\|x-x_0\|^2-\alpha$ and $g_i(x)=\|Bx\|^2+b_i^Tx-\beta_i$, $i=1,\ldots,m$, $A\in S^{n \times n}$, $B \in \mathbb{R}^{l \times n}$ with $l \in \mathbb{N}$, $a,x_0,b_i \in \mathbb{R}^n$ and $\gamma,\alpha,\beta_i \in \mathbb{R}$.
Suppose that the extended dimension condition (\ref{eq:EDC}) is satisfied.
%\[
%{\rm dim}{\rm Ker}(A -\lambda_{\rm min}(A)I_n)+{\rm dim}(\bigcap_{i=1}^mb_i^{\bot}) \ge n +1 .
%\]
Then, $$U(f,g_0,g_1,\ldots,g_m):=\{(f(x),g_0(x),g_1(x),\ldots,g_m(x)):x \in \mathbb{R}^n\}+\mathbb{R}_+^{m+2}$$ is a convex set.
\end{proposition}
\begin{proof}
As in the proof of Theorem \ref{prop:2.1}, we can assume without loss of generality that $A$ is not positive semidfinite.
Define $h$ by  $h(x)=\min_{x \in \mathbb{R}^n}\{f(x): \|x-x_0\|^2 \le \alpha+r, \, \|Bx\|^2+b_i^Tx \le \beta+s_i,\, i=1,\ldots,m\}$
 if $x \in D:=\{(r,s_1,\ldots,s_m): \|x-x_0\|^2 \le \alpha+r, \|Bx\|^2+b_i^Tx \le \beta_i+s_i \mbox{ for some } x \in \mathbb{R}^n\}$ and
$h(x)=+\infty$ if $x \notin D$. Using the same line of arguments as in Theorem  \ref{prop:2.1}, we can easily verify that $U(f,g_0,g_1,\ldots,g_m)={\rm epi}h$. Moreover, $h$ is convex if
the minimization problem  $$\min_{x \in \mathbb{R}^n}\{f(x)-\lambda_{\min}(A)\|x-x_0\|^2: \|x-x_0\|^2 \le \alpha+r, \|Bx\|^2+b_i^Tx \le \beta_i+s_i\}$$ attains its minimum at some $\overline{x} \in \mathbb{R}^n$
 with $\|\overline{x}-x_0\|^2=\alpha+r$ and $\|Bx\|^2+b_i^T\overline{x} \le \beta_i+s_i$.

Indeed, this optimization problem has a minimizer on the sphere. This follows from the fact that there exists $v \in \mathbb{R}^n \backslash\{0\}$ such that
\begin{equation}\label{eq:019}
v \in \big( \bigcap_{i=1}^mb_i^{\bot} \big) \cap {\rm Ker}(A -\lambda_{\rm min}(A)I_n) \cap {\rm Ker}(B).
\end{equation}

Otherwise, $\big( \bigcap_{i=1}^mb_i^{\bot} \big) \cap {\rm Ker}(A -\lambda_{\rm min}(A)I_n)\cap {\rm Ker}(B)=\{0\}$. Then it follows from our extended dimension condition,
  ${\rm dim}\big({\rm Ker}(A -\lambda_{\rm min}(A)I_n) \cap {\rm Ker}(B)\big) \ge s+1$, where $s$ is the dimension of ${\rm span}\{b_1,\ldots,b_m\}$, that
\begin{eqnarray*}
n+1  =  (s+1) +(n-s)
& \le & {\rm dim}({\rm Ker}(A -\lambda_{\rm min}(A)I_n\cap {\rm Ker}(B))+{\rm dim}(\bigcap_{i=1}^mb_i^{\bot}) \\
& = & {\rm dim}\big({\rm Ker}(A -\lambda_{\rm min}(A)I_n)\cap {\rm Ker}(B)+\bigcap_{i=1}^mb_i^{\bot} \big) \\
& & +{\rm dim} \big( \bigcap_{i=1}^mb_i^{\bot}  \cap {\rm Ker}(A -\lambda_{\rm min}(A)I_n\cap {\rm Ker}(B)) \big)\\
& \le &  n,
\end{eqnarray*}
which is impossible.

So, the same line of arguments as in Theorem \ref{prop:2.1} gives the desired conclusion.
\end{proof}

Recently, in \cite{beck-eldar}, the authors considered trust region problem with one additional linear inequality constraint:
$$(P_2) \ \ \ \min\{x^TAx+a^Tx:\|x-x_0\|^2 \le \alpha,\, b_1^Tx \le \beta_1\}$$
and showed that strong duality holds for $(P_1)$ whenever ${\rm dim}\big({\rm Ker}(A -\lambda_{\rm min}(A)I_n) \ge 2$.
Extending this, we consider the following quadratic optimizations with one additional convex quadratic constraint
$$(GP_2) \ \ \ \min\{x^TAx+a^Tx:\|x-x_0\|^2 \le \alpha,\, \|Bx\|^2+b_1^Tx \le \beta_1\}.$$
Following similar methods of proof of Section 3 and 4 and using the preceding proposition, we derive SDP relaxation and strong duality results for $(GP_1)$ under
the following dimension condition  ``${\rm dim}\big({\rm Ker}(A -\lambda_{\rm min}(A)I_n) \cap {\rm Ker}(B)\big) \ge 2$''.  However,
 it should be noted that, this dimension condition fails to be satisfied when $B$ has rank $n$ (the dimension of the underlying space). Indeed, in the case when $B$ has rank $n$, an example was
provided in \cite[{\rm Page 263 EX$_1$}]{Ye_Zhang} showing that the model $(GP_2)$ does not enjoy exact SDP relaxation as well as strong duality in general.

\begin{theorem}\label{th:999}
For problem $(GP_2)$, suppose that ${\rm dim}\big({\rm Ker}(A -\lambda_{\rm min}(A)I_n) \cap {\rm Ker}(B)\big) \ge 2$. Then, $(GP_2)$ admits exact SDP relaxation. Moreover, suppose
further that there exists $\overline{x}$ such that
$\|\overline{x}-x_0\|^2 < \alpha$ and $\|B\overline{x}\|^2+b_1^T\overline{x} < \beta_1$. Then, strong duality holds for problem $(GP_2)$, i.e.,
\begin{eqnarray*}
& & \min_{x \in \mathbb{R}^n}\{x^TAx+a^Tx:\|x-x_0\|^2 \le \alpha, \, \|Bx\|^2+b_1^Tx \le \beta_1\} \nonumber \\
&= &\max_{\lambda_0,\lambda_1 \ge 0} \min_{x \in \mathbb{R}^n}\{x^TAx+a^Tx+ \lambda_0(\|x-x_0\|^2-\alpha)+\lambda_1 (\|Bx\|^2+b_1^Tx \le \beta_1)\}.
\end{eqnarray*}
\end{theorem}
\begin{proof}
From Proposition \ref{prop:2.20} and the assumption ${\rm dim}\big({\rm Ker}(A -\lambda_{\rm min}(A)I_n) \cap {\rm Ker}(B)\big) \ge 2$, we see that $U(f,g_0,g_1)$ is convex where
$f(x)=x^TAx+a^Tx$, $g_0(x)=\|x-x_0\|^2-\alpha$ and $g_1(x)=\|Bx\|^2+b_1^Tx-\beta_1$. So, the first conclusion can be proved following similar line of argument as in Theorem \ref{th:value} while the second conclusion can be proved following similar line argument as in Theorem \ref{th:global} and
Corollary \ref{th:strong_duality}.
\end{proof}
\begin{remark}
{\rm A careful examination of the proof of above theorem shows that the conclusion of Theorem \ref{th:999} continues to hold for the quadratic problem $ \min\{x^TAx+a^Tx:\|x-x_0\|^2 \le \alpha,\, \|Bx\|^2+b_i^Tx \le \beta_i, i=1,\ldots,l\}$ under the condition ``${\rm dim}\big({\rm Ker}(A -\lambda_{\rm min}(A)I_n) \cap {\rm Ker}(B)\big) \ge l+1$". For simplicity, we only considered $(GP_2)$ with two constraints.}
\end{remark}

In the special case of ($GP_2$), where $B$ is the zero matrix, the preceding theorem 
reduces to \cite[Theorem 3.6]{beck-eldar} (see Corollary \ref{cor:1}).
%\begin{corollary}{\bf \cite[Theorem 3.6]{beck-eldar}}
% For problem $(P_1)$, suppose that ${\rm dim}\big({\rm Ker}(A -\lambda_{\rm min}(A)I_n)\big) \ge 2$ and suppose
% that there exists $\overline{x}$ such that
%$\|\overline{x}-x_0\|^2 < \alpha$ and $b_1^T\overline{x} < \beta_1$. Then, strong duality holds for problem $(P_1)$, i.e.,
%\begin{eqnarray*}
%& & \min_{x \in \mathbb{R}^n}\{x^TAx+a^Tx:\|x-x_0\|^2 \le \alpha, \, b_1^Tx \le \beta_1\} \nonumber \\
%&= &\max_{\lambda_0,\lambda_1 \ge 0} \min_{x \in \mathbb{R}^n}\{x^TAx+a^Tx+ \lambda_0(\|x-x_0\|^2-\alpha)+\lambda_1 (b_1^Tx \le \beta_1)\}.
%\end{eqnarray*}
%\end{corollary}
%\begin{proof}
%The conclusion follows by letting $B$ be the zero matrix in Theorem \ref{th:999}.
%\end{proof}

The following example illustrates that Theorem \ref{th:999} can be applied to some cases where $B$ is not a zero matrix.
\begin{example}
 Consider the following quadratic minimization problem
\begin{eqnarray*}
(P) & \min&   -x_1^2-x_2^2-x_3^2-2x_1 \\
& \mbox{ s.t. } & x_1^2+x_2^2+x_3^2+x_1 \le 1, \\
& & x_1^2+x_1 \le 0.
\end{eqnarray*}
This quadratic problem can be written as $(GP_2)$ with $f(x)=x^TAx+a^Tx$ with $A=-I_{3}$ and $a=(-2,0,0)$, $g_0(x)=\|x-x_0\|^2-\alpha$ with $x_0=(-\frac{1}{2},0,0)$, $\alpha=\frac{5}{4}$, and $g_1(x)=\|Bx\|^2+b_1^Tx-\beta_1$ with  $B=\left(\begin{array}{ccc}                                                                                         1 & 0                                                                                                          & 0 \\
0 & 0 & 0\\
0 & 0 & 0                                                                                                      \end{array}\right)
$, $b_1=(1,0,0)$ and $\beta_1=0$. One could verify that the strict feasibility condition is satisfied at $\overline{x}=(-\frac{1}{2},0,0)^T$ and
\[
  {\rm dim}\big({\rm Ker}(A -\lambda_{\rm min}(A)I_n) \cap {\rm Ker}(B)\big)= 2.
\]

% We now verify that the set $U(f,g_0,g_1,g_2)$ is convex where $g_0(x)=x_1^2+x_2^2+x_3^2-1$, $g_1(x)=x_1^2+4x_1$ and $g_2(x)=x_1^2-\frac{1}{4}$. To see this, note that
% \begin{eqnarray*}
% U(f,g_0,g_1,g_2)&=& \{(z,y_0,y_1,y_2): \exists x \in \mathbb{R}^3, \ z \ge -x_1^2-x_2^2-x_3^2, y_0 \ge x_1^2+x_2^2+x_3^2-1, \\
% & & \ \ \ \ \ \  y_1 \ge  x_1^2+4x_1, y_2\ge x_1^2-\frac{1}{4}  \} \\
% %& = & \{(z,y_0,y_1,y_2): \exists a \in \mathbb{R}^3,   a_i=x_i^2, i=1,2,3, z \ge -a_1-a_2-a_3, \\
% %& & \ \ \ \ \ \   y_0 \ge a_1+a_2+a_3-1, y_1 \ge  a_1+4x_1, y_2\ge a_1-\frac{1}{4}  \} \\
% & = & \{(z,y_0,y_1,y_2): \exists a \in \mathbb{R}^3,   a_i=x_i^2, i=1,2,3, x_1 \ge 0, z \ge -a_1-a_2-a_3, \\
% & & \ \ \ \ \ \   y_0 \ge a_1+a_2+a_3-1, y_1 \ge  a_1+4 \sqrt{a_1}, y_2\ge a_1-\frac{1}{4}  \}  \\
% & & \bigcup \{(z,y_0,y_1,y_2): \exists a \in \mathbb{R}^3,   a_i=x_i^2, i=1,2,3, x_1 \le 0, z \ge -a_1-a_2-a_3, \\
% & & \ \ \ \ \ \   y_0 \ge a_1+a_2+a_3-1, y_1 \ge  a_1-4 \sqrt{a_1}, y_2\ge a_1-\frac{1}{4}  \}  \\
% %&= & \bigcup \{(z,y_0,y_1,y_2): \exists a \in \mathbb{R}^3,   a_i=x_i^2, i=1,2,3,  z \ge -a_1-a_2-a_3, \\
% %& & \ \ \ \ \ \   y_0 \ge a_1+a_2+a_3-1, y_1 \ge  a_1-4 \sqrt{a_1}, y_2\ge a_1-\frac{1}{4}  \} \\
% &= &  \{(z,y_0,y_1,y_2): \exists a \in \mathbb{R}^3_+,   z \ge -a_1-a_2-a_3, \\
% & & \ \ \ \ \ \   y_0 \ge a_1+a_2+a_3-1, y_1 \ge  a_1-4 \sqrt{a_1}, y_2\ge a_1-\frac{1}{4}  \}.
% \end{eqnarray*}
% Then, as $x \mapsto x-4 \sqrt{x}$ is a convex function, it follows that $U(f,g_0,g_1,g_2)$ is convex.

Next, we show that strong duality and exact SDP relaxation hold. To see this, we note that, for any feasible point $x=(x_1,x_2,x_3)$, we have $-x_2^2-x_3^2 \ge x_1^2+x_1-1$ and $-1 \le x_1 \le 0$, and hence,
\[
-x_1^2-x_2^2-x_3^2-2x_1 \ge  -x_1-1 \ge -1.
\]
So, it can be easily seen that the optimal value of (P) is $-1$ and $(0,1,0)$ is a global minimizer. Let $\lambda_0=1$ and $\lambda_1=1$. Then,
\[
\min\{f(x)+\lambda_0g_0(x)+\lambda_1g_1(x)\}=\min\{x_1^2-1\}=-1=\min(P).
\]
So, the inequalities $\max(D) \le \min(SDRP) \le \min (P)$ imply that the strong duality and exact SDP relaxation hold.
\end{example}

Finally, we note that our approach and results in the present work suggest that the exact SDP-relaxation and strong duality may extend to multi-variate
 polynomial problems with a norm constraint and linear inequalities under an appropriate dimension condition. Moreover, it would be interesting to
 examine further potential applications of strong duality to robust optimization problems.  These will be our future research direction and will be examined in a forthcoming study.

%{\bf [Potentially works for polynomial?]}

\section*{Appendix: Technical Results}
For the sake of self-containment, in this Section, we provide known technical results on hidden convexity of quadratic systems, S-lemma and tractable classes of robust optimization.
\subsection*{Hidden Convexity of Quadratic Systems}
The basic and probably the most useful result on the joint-range
convexity of homogeneous quadratic functions, known as Dine's Theorem \cite{Dine}, states as follows:
\begin{lemma}{\bf (Dine's Theorem) \cite{Dine}}
Let  $A_1,
A_2\in S^n$. Then, the set
$\{(x^TA_1x,x^TA_2x):x \in\mathbb{R}^n\}$ is
convex.
\end{lemma}

Dine's theorem is known to fail for  three homogeneous in general.   Polyak \cite{Polyak} established  the following
joint-range convexity result for three homogeneous quadratic
functions under a positive definite condition on the matrices
involved.

\begin{lemma} \label{lemma:Polyak}{\bf (Polyak's Lemma \cite[Theorem 2.1]{Polyak})}
Let $n \ge 2$ and let  $A_1,
A_2,A_3 \in S^n$. Suppose that there exist
$\gamma_1,\gamma_2,\gamma_3 \in \mathbb{R}$ such that $\gamma_1 A_1+\gamma_2 A_2+\gamma_3A_3 \succ 0.$  Then
the set $\{(x^TA_1x,x^TA_2x,x^TA_3x):x \in\mathbb{R}^n\}$ is convex.
\end{lemma}

\subsection*{S-lemma and Approximate S-lemma}
Using Dine's Theorem, Yakubovich (cf \cite{S_lemma}) obtained the
following fundamental ${S}$-lemma which has played a key role in
many areas of control and optimization.
\begin{lemma}\label{th:S_lemma} {\bf (${S}$-lemma \cite{S_lemma})} Let $A_1,A_2 \in S^{n}$, $a_1,a_2 \in \mathbb{R}^n$ and $\alpha_1,\alpha_2 \in \mathbb{R}$. Suppose that there exists $x_0 \in \mathbb{R}^n$ such that $x_0^TA_2x_0+a_2^Tx_0+\alpha_2<0$. Then the
following statements are equivalent:\\
{\rm (i)} $x^TA_2x+a_2^Tx+\alpha_2 \le 0 \Rightarrow x^TA_1x+a_1^Tx+\alpha_1 \ge 0$ ; \\
{\rm (ii)} $(\exists \lambda \ge 0)$ $(\forall x \in \mathbb{R}^n)$
$ (x^TA_1x+a_1^Tx+\alpha_1 )+\lambda(x^TA_2x+a_2^Tx+\alpha_2 ) \ge 0.$
\end{lemma}
For a homogeneous quadratic system with multiple
convex quadratic constraints, Ben-Tal, Nemirovski and Roos \cite{BNR} derived
the following approximate S-lemma which provides an estimate between an associated quadratic optimization problem and its SDP relaxation.

\begin{lemma}{\bf (Approximate S-lemma \cite[Lemma A.6]{BNR})}
Let $R,H_0,H_1,\ldots,H_K$ be symmetric $(p \times p)$ matrices such that $H_i \succeq 0$, $i=1,\ldots,K$ and
$\sum_{k=0}^K\lambda_iH_i \succ 0,$
for some $\lambda_i \ge 0,$ $i=0,\ldots,K$. Consider the following quadratically constrained quadratic problem
\[
 (QCQ) \ \ \ \max_{y \in \mathbb{R}^p}\{y^TRy: y^TH_0y \le 1, y^TH_iy \le 1, i=1,\ldots,K\}
\]
and the semidefinite optimization problem
\[
(SDP) \ \ \ \min_{\mu_0,\ldots,\mu_K \ge 0}\{\sum_{i=0}^K \mu_i : \sum_{i=0}^K\mu_kH_k \succeq R\}.
\]
Then, $\max(QCQ) \le \min(SDP) \le \rho^2 \max(QCQ)$ where $\rho=\sqrt{2\log(6 \sum_{i=1}^K{\rm rank}H_k)}.$
\end{lemma}

\subsection*{Tractable Classes of Robust Optimization Problems}
The following tractable classes of robust optimization problems are known.

\subsubsection*{1. Robust least squares problems \cite{beck-eldar,robust_LS}}
Consider the following robust least squares programming problem:
\begin{eqnarray*}
(RLSP) &  \displaystyle  \min_{x \in \mathbb{R}^n}\max_{(A,a) \in \mathcal{U}} & \|Ax-a\|^2,
\end{eqnarray*}
where  $\mathcal{U}\subseteq \mathbb{R}^{k \times n} \times \mathbb{R}^{k}= \mathbb{R}^{k \times (n+1)}$, is an uncertainty set.
Then, (RLSP) can be equivalently rewritten as a semidefinite programming problem under the following two cases:
\begin{itemize}
 \item[{\rm (i)}] $\mathcal{U}$ is an ellipsoid (see \cite{robust_LS}), i.e., $\mathcal{U}=\{(A^{(0)},a^{(0)}) + \Delta: \Delta \in \mathbb{R}^{k \times (n+1)}, \|\Delta-\overline{\Delta}\|_F \le \rho\}$;
 \item[{\rm (ii)}] $k \ge 2$ and $\mathcal{U}$ is the intersection of two ellipsoids (see \cite{beck-eldar}), i.e,  % \[
$\mathcal{U}= \{(A^{(0)},a^{(0)}) + \Delta: \Delta \in \mathbb{R}^{k \times (n+1)}, \, {\rm Tr}(\Delta B_j \Delta) \le \rho_j^2,\  j = 1,2\}$ where $B_j \in S^{n \times n}$ satisfying $\gamma_1 B_1+\gamma_2 B_2 \succ 0$ for some $\gamma_1,\gamma_2 \ge 0$.
\end{itemize}

\subsubsection*{2. Robust second-order cone programming problems \cite{Beck0,BNR}}
Consider the following robust second order cone programming problem:
\begin{eqnarray*}
(RSOCP) & \min_{x \in \mathbb{R}^n} & a^Tx  \\
& \mbox{ s.t. } & \|B_ix-b_i\| \le d_i, \ \forall (B_i,b_i) \in \mathcal{U}_i,\   i=1,\ldots,m,
\end{eqnarray*}
where  $\mathcal{U}_i\subseteq \mathbb{R}^{k_i \times n} \times \mathbb{R}^{k_i}= \mathbb{R}^{k_i \times (n+1)}$, $i=1,\ldots,m$, is an uncertainty set.
Then, (RSOCP) can be equivalently rewritten as a semidefinite programming problem under the following two cases:
\begin{itemize}
 \item[{\rm (i)}] $\mathcal{U}_i$ is an ellipsoid (see \cite{BNR}), i.e., $\mathcal{U}_i=\{(B_i^{(0)},b_i^{(0)}) + \Delta_i: \Delta_i \in \mathbb{R}^{k_i \times (n+1)}, \|\Delta_i-\overline{\Delta}_i\|_F \le \rho_i\}$;
 \item[{\rm (ii)}] $\mathcal{U}_i$ is the intersection of at
 most $k$ many ellipsoids (see  \cite{Beck0}), i.e, $k_i=k$ with $k \in \mathbb{N}$ and % \[
$\mathcal{U}_i= \{(B_i^{(0)},b_i^{(0)}) + \Delta: \Delta \in \mathbb{R}^{k \times (n+1)}, \, \|C_j \Delta^T\|_F^2 \le \rho_j^2,\  j = 1,\ldots,k\}$,
where $C_j \in \mathbb{R}^{(n+1) \times (n+1)}$ such that there exist $\mu_j \in \mathbb{R}$ such that $\sum_{j=1}^k\mu_jC_j^TC_j \succ 0.$
\end{itemize}

\end{document}